\documentclass[12pt, reqno]{amsart}
\usepackage{hyperref}
\usepackage{graphicx}
\usepackage{amsmath}
\usepackage{amsthm}
\usepackage{amssymb,bbm}%
\usepackage[numbers, square]{natbib}

  \evensidemargin
\oddsidemargin
\newcommand{\E}{\mathbb E}

\newcommand{\R}{\mathbb{R}}
\newcommand{\N}{\mathbb{N}}

\renewcommand{\P}{\mathbb{P}}

\newcommand{\simn}{\underset{n \to \infty}{\sim}}
\newcommand{\cP}{\mathcal{P}}
\newcommand{\cF}{\mathcal{F}}
\newcommand{\aff}{\mathop{\mathrm{aff}}\nolimits}

\newcommand{\Vol}{\mathop{\mathrm{Vol}}\nolimits}

\newcommand{\spec}{\mathop{\mathrm{Spec}}\nolimits}

\newcommand{\Nor}{\mathop{\mathrm{Nor}}\nolimits}

\newcommand{\sgn}{\mathop{\mathrm{sgn}}\nolimits}
\newcommand{\conv}{\mathop{\mathrm{conv}}\nolimits}

\newcommand{\zon}{\mathop{\mathrm{Zon}}\nolimits}

\newcommand{\eps}{\varepsilon}
\newcommand{\eqdistr}{\stackrel{d}{=}}

\newcommand{\toas}{\overset{a.s.}{\underset{n\to\infty}\longrightarrow}}

\newcommand{\bsl}{\backslash}
\newcommand{\ind}{\mathbbm{1}}

\newcommand{\dd}{{\rm d}}
\newcommand{\eee}{{\rm e}}

\theoremstyle{plain}
\newtheorem{theorem}{Theorem}[section]
\newtheorem{lemma}[theorem]{Lemma}

\newtheorem{proposition}[theorem]{Proposition}

\theoremstyle{definition}

\newtheorem{example}[theorem]{Example}

\theoremstyle{remark}
\newtheorem{remark}[theorem]{Remark}

\begin{document}

\author{Zakhar Kabluchko}
\address{Zakhar Kabluchko: Institut f\"ur Mathematische Stochastik,
Westf\"alische Wilhelms-Universit\"at M\"unster,
Orl\'eans--Ring 10,
48149 M\"unster, Germany}
\email{zakhar.kabluchko@uni-muenster.de}

\author{Dmitry Zaporozhets}
\address{Dmitry Zaporozhets: St.\ Petersburg Department of Steklov Mathematical Institute,
Fontanka~27,
191011 St.\ Petersburg,
Russia}
\email{zap1979@gmail.com}

\title[Volume of the Gaussian polytope]{Expected volumes of Gaussian polytopes, external angles,  and multiple order statistics}
\thanks{}
\begin{abstract}
Let $X_1,\ldots,X_n$ be a standard normal sample in $\R^d$.
We compute exactly the expected volume of the Gaussian polytope $\mathrm{conv}\, [X_1,\ldots,X_n]$, the symmetric Gaussian polytope $\mathrm{conv}\, [\pm X_1,\ldots,\pm X_n]$, and the Gaussian zonotope $[0,X_1]+\ldots+[0,X_n]$ by exploiting their connection to the regular simplex, the regular crosspolytope, and the cube with the aid of Tsirelson's formula. The expected volumes of these random polytopes are given by essentially the same expressions as the intrinsic volumes and external angles of the regular polytopes. For all these quantities, we obtain asymptotic formulae which are more precise than the results which were known before.
More generally, we determine the expected volumes of some heteroscedastic random polytopes including
$
\mathrm{conv}\,[l_1X_1,\ldots,l_nX_n]
$
and
$
\mathrm{conv}\, [\pm l_1 X_1,\ldots, \pm l_n X_n],
$
where $l_1,\ldots,l_n\geq 0$ are parameters, and the intrinsic volumes of the corresponding deterministic polytopes. Finally, we relate the $k$-th intrinsic volume of the regular simplex $S^{n-1}$  to the expected maximum of independent standard Gaussian random variables $\xi_1,\ldots,\xi_n$ given that the maximum has multiplicity $k$. Namely, we show that
$$
V_k(S^{n-1})
= \frac {(2\pi)^{\frac k2}} {k!}   \cdot \lim_{\eps\downarrow 0} \eps^{1-k} \E [\max\{\xi_1,\ldots,\xi_n\} \ind_{\{\xi_{(n)} - \xi_{(n-k+1)}\leq \eps\}}],
$$
where $\xi_{(1)} \leq \ldots \leq \xi_{(n)}$ denote the order statistics. A similar result holds for the crosspolytope if we replace $\xi_1,\ldots,\xi_n$ by their absolute values.
\end{abstract}

\maketitle

\section{Introduction}\label{sec:introduction}
\subsection{Gaussian polytopes}
Let $X_1,\ldots,X_n$ be random points sampled independently from the standard normal distribution on $\R^d$. The random polytopes
$$
\cP_{n,d} := \conv [X_1,\ldots,X_n]\;\;\; \text{ and } \;\;\; \cP_{n,d}^\pm := \conv [\pm X_1,\ldots,\pm X_n],
$$
where $\conv [\ldots]$ denotes the convex hull,
are called the \textit{Gaussian polytope} and the \textit{symmetric Gaussian polytope}, respectively.

We shall be interested in the expected volumes of these polytopes. Let $\kappa_d = \pi^{\frac d2}/ \Gamma(\frac d2+1)$ be the volume of the $d$-dimensional unit ball $\mathbb B_d$. We write $\varphi(t)$ and $\Phi(t)$ for the standard normal density and distribution function:
\begin{equation}\label{eq:def_phi}
\varphi(t) = \frac 1 {\sqrt{2\pi}}\eee^{-t^2/2},
\quad
\Phi(t)=\int_{-\infty}^t \varphi(s)\dd s.
\end{equation}
\begin{theorem}\label{theo:exp_polytopes}
The expected volumes of $\cP_{n,d}$ and $\cP_{n,d}^\pm$ are given by
\begin{align}
\E \Vol_d(\cP_{n,d})
&=
 \frac{\kappa_d \, n!}{d!(n-d-1)!}\int_{-\infty}^{\infty} \Phi^{n-d-1}(t) \varphi^{d+1}(t)\dd t,
 \; &1 \leq d < n&, \label{eq:exp_vol_P_n_d}\\
\E \Vol_d(\cP_{n,d}^\pm)
&=
 \frac{\kappa_d \, n!}{d!(n-d-1)!} \int_{0}^{\infty} (2\Phi(t)-1)^{n-d-1} (2\varphi(t))^{d+1}\dd t,
  \; &1 \leq d \leq  n&.
\label{eq:exp_vol_P_n_d_symm}
\end{align}
\end{theorem}

Formula~\eqref{eq:exp_vol_P_n_d} was obtained by Efron~\cite{efron} for $d=2$ and was stated on the last page of~\cite{efron} for general $d$, whereas formula~\eqref{eq:exp_vol_P_n_d_symm} seems to be new.    Both formulae~\eqref{eq:exp_vol_P_n_d} and~\eqref{eq:exp_vol_P_n_d_symm} will be derived in Sections~\ref{subsec:simplex_tsirelson} and~\ref{subsec:cross_tsirelson}, below.  Additionally, for the \textit{Gaussian polytope with $0$} that is defined by $\cP_{n,d}^{(0)} := \conv [0, X_1,\ldots,X_{n}]$, we shall prove that
\begin{equation}\label{eq:exp_vol_P_n_d_zero}
\E \Vol_d (\cP_{n,d}^{(0)})
=
\frac {\binom nd} {2^{n- \frac d2} \Gamma(\frac d2+1) } + \frac{\kappa_d n!}{d!(n-d-1)!} \int_0^{\infty}  \Phi^{n-d-1}(t)\varphi^{d+1}(t)\dd t
\end{equation}
for $1\leq d\leq n$. More generally, we shall derive exact formulae for the expected volumes of the ``heteroscedastic'' polytopes $
\conv[l_1X_1,\ldots,l_nX_n]
$
and
$
\conv [l_1^+ X_1, -l_1^- X_1, \ldots, l_n^+ X_n, -l_n^- X_n],
$
where $l_1,\ldots,l_n$ (respectively, $l_1^+, l_1^-, \ldots, l_n^+, l_n^-$) are non-negative parameters; see Section~\ref{sec:non_regular}.

Early results on Gaussian polytopes are due to R\'enyi and Sulanke~\cite[\S4]{renyi_sulanke1},  Raynaud~\cite{raynaud} and Geffroy~\cite{geffroy}. In the case $n=d+1$, the Gaussian polytope $\cP_{d+1,d}$ becomes a simplex  and the distribution of its volume was characterized explicitly by Miles~\cite{miles}. For explicit computation of other functionals of Gaussian polytopes (for example, the expected number of faces and absorption probabilities),  see~\cite{raynaud,AS92,baryshnikov_vitale,HMR04,hug_reitzner,kabluchko_zaporozhets_gauss_polytope}. Central limit theorem for the volume  and the number of faces of $\cP_{n,d}$ was established by Barany and Vu~\cite{barany_vu}. The asymptotic variances of these quantities and the scaling limit of the boundary were characterized by Calka and Yukich~\cite{calka_yukich}. Convex hulls of points chosen uniformly in a ball or on the sphere were studied by Kingman~\cite{kingman} (who computed the expected volume in the case of the simplex), Raynaud~\cite{raynaud} (who computed the asymptotics of the expected number of facets), and Affentranger~\cite{affentranger_ball} (who computed the expected volume for general $n$ and $d$).
Convex hulls of i.i.d.\ samples from more general beta-type spherically symmetric distributions were studied by Miles~\cite{miles} (in the case of the simplex) and by~\citet{affentranger} (who derived the asymptotics of the volume for general $n$ and $d$). There are also many references on convex hulls of points chosen uniformly in a convex body, but we shall focus on the Gaussian case here.

\subsection{Regular polytopes}
The above formulae~\eqref{eq:exp_vol_P_n_d} and~\eqref{eq:exp_vol_P_n_d_symm} look very similar to the formulae for the external angles and intrinsic volumes of the regular simplex and the regular crosspolytope, respectively. We shall recall the definition of the external angles and intrinsic volumes in Section~\ref{subsec:intrinsic_and external}, below. Since for regular polytopes these quantities differ by multiplicative constants, we state only the formulae for the external angles.

Let $e_1,\ldots,e_n$ be the standard orthonormal basis in $\R^n$. The  $(n-1)$-dimensional \textit{regular simplex} $S^{n-1}$ and the $n$-dimensional \textit{regular crosspolytope} $C^n$ are defined by
$$
S^{n-1} = \conv [e_1,\ldots,e_n] \;\;\; \text{ and } \;\;\; C^{n} = \conv [\pm e_1,\ldots, \pm e_n],
$$
respectively. Then, the external angles of $S^{n-1}$ and $C^{n}$ at any $k$-dimensional face $F_k$ are given by
\begin{align}
\gamma(F_k, S^{n-1})
&=
(2\pi)^{k/2} \sqrt{k+1} \int_{-\infty}^{\infty} \Phi^{n-k-1}(t) \varphi^{k+1}(t)\dd t,
&0&\leq k < n,
\label{eq:ext_angle_simplex}\\
\gamma(F_k, C^{n})
&=
(2\pi)^{k/2}  \sqrt{k+1} \int_{0}^{\infty} (2\Phi(t)-1)^{n-k-1} \varphi^{k+1}(t)\dd t,
&0&\leq k < n.\label{eq:ext_angle_cross_betke}
\end{align}
Formula~\eqref{eq:ext_angle_simplex} can be found in the work of Ruben~\cite{ruben,ruben_moments}. A more transparent proof was provided by~\citet{hadwiger}. It is also possible to relate the external angles of regular simplices to the volumes of regular spherical simplices which were computed in the book of~\citet[Satz 3 on p.~283]{boehm_hertel_book}. Formula~\eqref{eq:ext_angle_cross_betke} was obtained by~\citet[Lemma~2.1]{betke_henk}.

Similarly to the above, the counterpart of~\eqref{eq:exp_vol_P_n_d_zero} is a formula  for the external angle of the rectangular simplex $S^n_0:= \conv[0,e_1,\ldots,e_n]$ at any $k$-dimensional face $F_k$ not containing $0$, for example $F_k = \conv[e_1,\ldots,e_{k+1}]$:
\begin{equation}\label{eq:ext_angle_simplex_zero}
\gamma(F_k, S^{n}_0)
=
(2\pi)^{k/2}  \sqrt{k+1}   \int_0^{\infty}  \Phi^{n-k-1}(t)\varphi^{k+1}(t)\dd t.
\end{equation}
This formula was stated in~\cite[Theorem~2.2]{betke_henk} in terms of the intrinsic volumes.

The methods by which~\eqref{eq:exp_vol_P_n_d} and~\eqref{eq:ext_angle_simplex} were derived in the literature do not explain the similarity between these formulae. In Section~\ref{sec:sudakov_tsirelson} we shall establish an equivalence between~\eqref{eq:exp_vol_P_n_d}, \eqref{eq:exp_vol_P_n_d_symm}, \eqref{eq:exp_vol_P_n_d_zero} and \eqref{eq:ext_angle_simplex}, \eqref{eq:ext_angle_cross_betke}, \eqref{eq:ext_angle_simplex_zero} via a theorem due to Tsirelson~\cite{tsirelson2}.

\subsection{Extreme values and asymptotics}
Let $\xi_1,\ldots,\xi_n$ be independent one-dimensional standard Gaussian variables.  In dimension $d=1$, the Gaussian polytopes are just intervals, namely
\begin{equation}\label{eq:polytopes_one_dim}
\cP_{n,1} \eqdistr [\xi_{(1)},\xi_{(n)}],
\quad
\cP_{n,1}^\pm \eqdistr [-|\xi|_{(n)},|\xi|_{(n)}],
\end{equation}
where $\xi_{(1)} = \min\{\xi_1,\ldots,\xi_n\}$, $\xi_{(n)} = \max\{\xi_1,\ldots,\xi_n\}$ and $|\xi|_{(n)} = \max\{|\xi_1|,\ldots,|\xi_n|\}$. For arbitrary $d\in\N$, one can view $\cP_{n,d}$ as a generalized sample range, whereas the vertices of $\cP_{n,d}$ can be thought of as the generalized sample extremes.
Historically, this connection to extreme-value theory was one of the motivations for studying random convex hulls; see~\cite{MCR10} for a recent work in this direction. It follows from~\eqref{eq:polytopes_one_dim} that for $d=1$,
$$
\E \Vol_1 (\cP_{n,1}) = 2 \E \max\{\xi_1,\ldots,\xi_n\},
\quad
\E \Vol_1 (\cP_{n,1}^{\pm}) = 2 \E \max\{|\xi_1|,\ldots,|\xi_n|\}.
$$
Since the distribution function of $\xi_{(n)}$ is $\Phi^n(t)$ and since some power of $\Phi$ appears in the exact formula for $\E \Vol_d (\cP_{n,d})$, see~\eqref{eq:exp_vol_P_n_d}, one may ask whether a similar connection between $\E \Vol_d (\cP_{n,d})$ and \textit{one-dimensional} extremes exists for arbitrary $d$. This question will be affirmatively answered in Section~\ref{sec:multiple}, where we shall express $\E \Vol_d (\cP_{n,d})$ (or, which is almost the same, the intrinsic volumes of the regular simplex $S^{n-1}$) through the expected maximum in a normal sample $\xi_1,\ldots,\xi_n$ on the event that the maximum is attained with multiplicity $d$.

Let us now turn to the asymptotic results. Throughout, we consider the regime in which $n\to\infty$, while the dimension $d$ stays fixed.  We write $a_n\sim b_n$ if $\lim_{n\to\infty} a_n/b_n = 1$. Then,
\begin{equation}\label{eq:asympt_exp_vol}
\E \Vol_d(\cP_{n,d})
\sim
\kappa_d  (2\log n)^{d/2}
\;\;\; \text{ and } \;\;\;
\E \Vol_d(\cP_{n,d}^\pm)
\sim
\kappa_d  (2\log n)^{d/2}.
\end{equation}
The first formula is due to Affentranger~\cite[Theorem~4]{affentranger}. We did not found a reference for the second one, but anyway both formulae are not surprising since it is known~\cite{geffroy} that
\begin{equation}\label{eq:haus_dist_conv_hull_ball}
d_{\text{Haus}}\left(\frac 1 {\sqrt{2\log n}} \cP_{n,d}, \mathbb B_d\right) \toas 0,
\quad
d_{\text{Haus}}\left(\frac 1 {\sqrt{2\log n}} \cP_{n,d}^\pm, \mathbb B_d\right) \toas 0,
\end{equation}
where $d_{\text{Haus}}$ is the Hausdorff distance and $\mathbb B_d$ is the $d$-dimensional unit ball.
A result much more general than~\eqref{eq:asympt_exp_vol} and~\eqref{eq:haus_dist_conv_hull_ball} holds universally for the convex hull of $n$ independent realisations of any bounded Gaussian process; see Davydov~\cite{davydov}. It is therefore of interest to obtain a more refined asymptotics than that in~\eqref{eq:asympt_exp_vol} since such result would say something about the deviation between the ball of radius $\sqrt{2\log n}$ and the polytopes $\cP_{n,d}$ and $\cP_{n,d}^\pm$. Calka and Yukich~\cite[Theorem~1.4]{calka_yukich} proved\footnote{In fact, they considered a poissonized version of $\cP_{n,d}$; see Section~\ref{subsec:poissonized}.} that
$$
\E \Vol_d(\cP_{n,d})= \kappa_d (2\log n)^{d/2} \left(1-\frac{d\log \log n}{4 \log n} + O\left(\frac 1 {\log n}\right)\right),
\quad n\to\infty.
$$

We shall prove a more precise asymptotic formula. Let us first consider the case $d=1$. It is a standard result of the extreme-value theory~\cite[Theorem~1.5.3 on p.~14]{leadbetter_etal_book} and~\cite[Theorem 1.8.3 on p.~28]{leadbetter_etal_book} that both $u_n (\xi_{(n)} - u_n)$ and $u_{2n}(|\xi|_{(n)} - u_{2n})$ converge weakly to the Gumbel law with distribution function $\eee^{-\eee^{-z}}$, $z\in\R$, where  $u_n$ is the sequence given by
\begin{equation}\label{eq:def_u_n_0}
u_n = \sqrt{2\log n}  - \frac{\frac 12 \log \log n +\log (2\sqrt{\pi})}{\sqrt{2\log n}}.
\end{equation}
Observe that $|\xi|_{(n)}= \max\{\xi_1,\ldots,\xi_n, -\xi_1,\ldots,\xi_n\}$ should be normalized in the same way as $\xi_{(2n)}$, which expresses the fact that the $2n$ standard Gaussian variables $\pm \xi_1,\ldots, \pm \xi_n$ become asymptotically independent.
Taking the expectation (which can easily be justified, see~\cite{pickands_moments} for a general result) and noting that the expectation of the Gumbel distribution is the Euler constant $\gamma= -\Gamma'(1)$, we obtain
\begin{align*}
\E \Vol_1 (\cP_{n,1}) &= 2 \E  \xi_{(n)} = 2 (u_n + \gamma u_n^{-1}) +  o(u_n^{-1}),\\
\E \Vol_1 (\cP_{n,1}^{\pm}) &= 2 \E |\xi|_{(n)} = 2 (u_{2n} + \gamma u_{2n}^{-1}) +  o(u_{2n}^{-1}),
\end{align*}
as $n\to\infty$.
The next theorem generalizes these relations to arbitrary dimension $d$.
\begin{theorem}\label{theo:asympt_E_Vol}
If $d\in\N$ stays fixed and $n\to\infty$, then
\begin{align}
\E \Vol_d(\cP_{n,d})
&=
\kappa_d  u_n^d + d\kappa_d  u_n^{d-2} \left(\gamma - \sum_{j=2}^d \frac 1j\right) + o(u_n^{d-2}),\label{eq:E_Vol_asympt_simplex}\\
\E \Vol_d(\cP_{n,d}^\pm)
&=
\kappa_d  u_{2n}^d + d\kappa_d  u_{2n}^{d-2} \left(\gamma - \sum_{j=2}^d \frac 1j\right) + o(u_{2n}^{d-2}). \label{eq:E_Vol_asympt_cross}
\end{align}
\end{theorem}

The first formula can be interpreted as follows. In the first approximation, $\cP_{n,d}$ is close to the $d$-dimensional ball of radius $u_n$ centered at the origin. The volume of this ball is $\kappa_d u_n^d$.  The boundary of $\cP_{n,d}$ is close to the boundary of this ball,  however there is a small deviation between both boundaries. Consider some point on the sphere, for example $p_n := (u_n,0,\ldots, 0)$. Let $q_n= (u_n+h_n, 0,\ldots,0)$ be the point where the ray passing through $0$ and $p_n$ intersects the boundary of $\cP_{n,d}$.
So,  $h_n$ is the signed deviation between the boundaries of the ball and $\cP_{n,d}$ in the direction of $p_n$.  The $(d-1)$-dimensional volume of the $(d-1)$-dimensional sphere of radius $u_n$ equals $d\kappa_d u_n^{d-1}$. Hence, $\E \Vol_d (\cP_{n,d})$ should be close to $\kappa_d  u_n^d + d \kappa_d u_n^{d-1} \E h_n$. Comparing this with~\eqref{eq:E_Vol_asympt_simplex}, we obtain
$$
\E h_n = \frac {1+o(1)} {u_n} \left(\gamma -\sum_{j=2}^d \frac 1j\right).
$$
For example, in dimension $d=1$, it is known from extreme-value theory that $u_n h_n$ converges weakly to the Gumbel distribution whose expectation equals $\gamma$. For arbitrary $d\in\N$, Calka and Yukich~\cite[Theorem~1.2]{calka_yukich} showed\footnote{Actually, they considered a poissonized version of $\cP_{n,d}$, but~\eqref{eq:E_Vol_asympt_simplex} and the subsequent considerations remain valid for the poissonized version; see Section~\ref{subsec:poissonized}. The term $(d-1) \log \sqrt{2\pi}$ appears because Calka and Yukich~\cite[Eq.~(1.4)]{calka_yukich} used a normalization sequence which was  slightly different from $u_n$.} that $- u_n h_n + (d-1) \log \sqrt{2\pi}$ converges weakly to the marginal height distribution of the so-called Burgers festoon over $\R^{d-1}$. Thus, the expected height of the Burgers festoon over any point in $\R^{d-1}$ equals $\sum_{j=2}^d \frac 1j - \gamma + (d-1) \log \sqrt{2\pi}$.

Let us finally mention an equivalent version of Theorem~\ref{theo:asympt_E_Vol} in terms of the intrinsic volumes (to be defined in Section~\ref{subsec:intrinsic_and external}). The proof will be given in Sections~\ref{subsec:simplex_tsirelson} and~\ref{subsec:cross_tsirelson}.
\begin{theorem}
Let $d\in\N$ be fixed. The intrinsic volumes of the regular simplex $S^{n-1}=\conv[e_1,\ldots,e_n]$ and the regular crosspolytope $C^n=\conv[\pm e_1,\ldots,\pm e_n]$ satisfy, as $n\to\infty$,
\begin{align*}
V_d(S^{n-1}) &= \frac {(2\pi)^{d/2}}{d!} \left(u_n^d + d u_n^{d-2} \left(\gamma - \sum_{j=2}^d \frac 1j\right) + o(u_n^{d-2})\right)\sim \frac {(4\pi \log n)^{d/2}} {d!},\\
V_d(C^{n}) &= \frac {(2\pi)^{d/2}}{d!} \left(u_{2n}^d + d u_{2n}^{d-2} \left(\gamma - \sum_{j=2}^d \frac 1j\right) + o(u_{2n}^{d-2})\right)\sim \frac {(4\pi \log n)^{d/2}} {d!}.
\end{align*}
\end{theorem}
For $d=1$ (in which case $V_1$ is, up to a constant factor, the so-called mean width), these formulae were obtained by Finch~\cite{finch_simplex, finch_cross}.

\subsection{Organization of the paper}
In Section~\ref{sec:sudakov_tsirelson} we shall recall theorems due to Sudakov~\cite{vS76} and Tsirelson~\cite{tsirelson2} and use the latter to show that the formulae for the expected volumes of $\cP_{n,d}$, respectively $\cP_{n,d}^\pm$,  are equivalent to the formulae for the external angles (or intrinsic volumes) of the regular simplex, respectively crosspolytope.
In Section~\ref{sec:non_regular} we shall compute exactly the expected volumes of some generalized Gaussian polytopes including those of the form $\conv [l_1X_1,\ldots,l_nX_n]$, where $l_1,\ldots,l_n\geq 0$ are deterministic parameters.
In Section~\ref{sec:multiple} we shall relate the intrinsic volumes of the regular simplex/crosspolytope to the expected maximum in a normal sample on  the event that the maximum is attained with given multiplicity. Finally, in Section~\ref{sec:asymptotics} we shall prove an asymptotic formula for the integrals of the form $\int_{-\infty}^{\infty} \varphi^{\alpha}(t) \Phi^n(t)\dd t$, as $n\to\infty$,  which implies Theorem~\ref{theo:asympt_E_Vol}.

\section{Sudakov's and Tsirelson's theorems}\label{sec:sudakov_tsirelson}
\subsection{Intrinsic volumes and external angles}\label{subsec:intrinsic_and external}
For a bounded convex set $T\subset\R^n$  the \textit{intrinsic volumes} $V_0(T), \ldots, V_n(T)$ are defined by the Steiner formula
\begin{equation}\label{eq:steiner}
\Vol_n(T+r\mathbb B_n)=\sum_{k=0}^n \kappa_{n-k} V_k(T) r^{n-k}, \quad r\geq 0,
\end{equation}
where $\mathbb B_n$ denotes the $n$-dimensional unit ball and $\kappa_k=\pi^{k/2}/\Gamma(\frac k 2 +1)$ is the volume of $\mathbb B_k$.  An alternative definition is given by \textit{Kubota's formula}~\cite[p.~222]{SW08} which states that
\begin{equation}\label{eq:kubota}
V_m(T)= \binom nm\frac{\kappa_n}{\kappa_m \kappa_{n-m}}\E \Vol_m (T | L),
\end{equation}
where  $T|L$  is the projection of $T$ onto a uniformly chosen random $m$-dimensional linear subspace $L\subset \R^d$.
In particular, $V_d(T)= \Vol_d(T)$ is the $d$-dimensional volume, $V_{d-1}(T)$ is half the surface area, $V_0(T)=1$, while  $V_1(T)$ coincides with the mean width of $T$, up to a constant factor. Finally, if $T=P$ is a convex polytope, then~\cite[Eq.~(14.14) on p.~607]{SW08}
\begin{equation}\label{eq:intrinsic_volume_external_angles}
V_{m}(P) = \sum_{F\in \cF_m(P)} \Vol_{m}(F) \gamma(F, P),
\end{equation}
where $\cF_m(P)$ is the set of $m$-dimensional faces of $P$, and $\gamma(F, P)$ is the external angle at the face $F\in \cF_{m}(P)$.

Let us recall the definition of the external angles. The relative interior of the face $F\in \cF_m(P)$ is the interior of $F$ taken with respect to the affine hull $\aff (F)$ as the ambient space. The normal cone of the polytope $P\subset \R^n$ at its face $F$ is defined as
$$
\Nor(P,F) = \{u\in\R^n\colon \langle  u, p-x\rangle \leq 0 \text{ for all } p\in P\},
$$
where $x$ is any point in the relative interior of $F$.  The angle of this cone, referred to as the \textit{external angle} of $P$ at its face $F$,  is defined as
$$
\gamma(F,P) := \P[Z\in \Nor(P,F)],
$$
where $Z$ is a random vector having the standard normal distribution (or, more generally, any spherically symmetric distribution) on the linear hull of $\Nor(P,F)$.

\subsection{Sudakov's and Tsirelson's theorems}\label{subsec:sudakov_thm}
Let  $e_1,\ldots,e_n$ be the standard orthonormal basis in $\R^n$. The \textit{isonormal process} $\{\xi(x)\colon x\in \R^n\}$ is defined by
$$
\xi(x_1e_1+ \ldots + x_ne_n) = \sum_{i=1}^n \xi_i x_i, \quad x= (x_1,\ldots,x_n)\in\R^n,
$$
where $\xi_1,\ldots, \xi_n$ are i.i.d.\ standard normal random variables. Note that $\xi$ is a centered Gaussian process with covariance function $\E [\xi(x)\xi(y)] = \langle x, y\rangle$, $x,y\in\R^n$.
The next formula is due to Sudakov~\cite{vS76}.
\begin{theorem}[Sudakov]\label{theo:sudakov}
For every compact convex set $T\subset \R^n$ it holds that
\begin{equation}\label{2041}
V_1(T)= \sqrt{2\pi}\,\E\,\sup_{x\in T} \xi(x).
\end{equation}
\end{theorem}

Tsirelson~\cite{tsirelson2} generalized Sudakov's formula to all intrinsic volumes as follows. Consider $d$ independent copies  $\{\xi_i(x)\colon x\in \R^n\}$, $1\leq i \leq d$, of the isonormal process. The $d$-dimensional spectrum of a compact convex set $T\subset \R^n$ is a random set
$$
\spec_d T = \{(\xi_1(x), \ldots, \xi_d(x))\colon x\in T\}\subset \R^d.
$$
\begin{theorem}[Tsirelson]\label{theo:tsirelson_spectrum}
For every $d\in \{1,\ldots,n\}$ and every compact convex set $T\subset \R^n$ it holds that
\begin{equation}\label{2042}
V_d(T)= \frac{(2\pi)^{d/2}}{d!\kappa_d} \E\,\Vol_d(\spec_d T).
\end{equation}
\end{theorem}
In fact, both theorems were established in an arbitrary Hilbert space $H$, but in the present paper we need only the finite-dimensional case $H=\R^n$.
For further results related to Tsirelson's theorem, we refer to the work of Vitale~\cite{vitale01,vitale08,vitale10} and Chevet~\cite{chevet}. The proof of Tsirelson's formula is based on the observation that the projection of $T$ onto a random uniform $d$-dimensional linear subspace differs from $\spec_d T$ by an explicit random linear transformation. Similar idea was used by~\citet{baryshnikov_vitale} to prove\footnote{Baryshnikov and Vitale~\cite{baryshnikov_vitale} stated their result when $P$ is a regular simplex or a cube, but their proof applies to general polytopes.} the following

\begin{theorem}[Baryshnikov and Vitale]
If $P\subset \R^n$ is a polytope, then the number of $k$-faces of a projection of $P$ onto a  random, uniformly distributed linear subspace of dimension $d$ has the same distribution as the number of $k$-faces of $\spec_d P$.
\end{theorem}

\subsection{Example: Regular simplex and the Gaussian polytope}\label{subsec:simplex_tsirelson}
We would like to apply Tsirelson's Theorem~\ref{theo:tsirelson_spectrum} to the regular simplex
$$
T := S^{n-1} = \conv[e_1,\ldots,e_n] \subset \R^n.
$$
It is easy to see that $\spec_d S^{n-1}$ is the convex hull of $n$ i.i.d.\ random vectors $X_1,\ldots,X_n$ having a standard Gaussian distribution on $\R^d$, so that we can identify $\spec_d S^{n-1}$ with the Gaussian polytope $\cP_{n,d}$.  Theorem~\ref{theo:tsirelson_spectrum} yields the relation
\begin{equation}\label{eq:tsirelson_simplex}
V_d(S^{n-1})= \frac{(2\pi)^{d/2}}{d!\kappa_d} \E\,\Vol_d(\cP_{n,d}), \quad  1\leq d < n.
\end{equation}
To compute $V_d(S^{n-1})$, note that the number of $d$-faces of $S^{n-1}$ is $\binom {n}{d+1}$, and any such face has $d$-volume $\sqrt{d+1}/d!$ (see, e.g.,\ Lemma~\ref{lem:volume_simplex}, below). Using~\eqref{eq:intrinsic_volume_external_angles} together with the formula for the external angle~\eqref{eq:ext_angle_simplex} yields
\begin{equation}\label{eq:V_k_S_n}
V_d(S^{n-1}) =   (2\pi)^{d/2} \binom{n}{d+1}\frac{d+1}{d!} \int_{-\infty}^{\infty} \Phi^{n-d-1}(t) \varphi^{d+1}(t)\dd t, \quad 0\leq d <  n.
\end{equation}
Combining Theorem~\ref{theo:asympt_E_Vol} with~\eqref{eq:tsirelson_simplex}, we obtain the asymptotic formula
$$
V_d(S^{n-1}) = \frac {(2\pi)^{d/2}}{d!} \left(u_n^d + d u_n^{d-2} \left(\gamma - \sum_{j=2}^d \frac 1j\right) + o(u_n^{d-2})\right)\sim \frac{(4\pi \log n)^{d/2}}{d!},
$$
where $u_n$ is given by~\eqref{eq:def_u_n_0}.
Here, $d\in\N$ is fixed and $n\to\infty$.
By~\eqref{eq:tsirelson_simplex}, the exact formula~\eqref{eq:V_k_S_n} is equivalent to Efron's~\cite{efron} formula
\begin{equation}\label{eq:V_k_S_n_tsirelson}
\E \Vol_d (\cP_{n,d})
= \frac{\kappa_d \, n!}{d!(n-d-1)!} \int_{-\infty}^{\infty} \Phi^{n-d-1}(t) \varphi^{d+1}(t)\dd t, \quad 1\leq d < n.
\end{equation}

Using Kubota's formula~\eqref{eq:kubota} and the fact that an orthogonal projection of i.i.d.\ Gaussian sample is again i.i.d.\ Gaussian, one obtains formulae for expectations of arbitrary intrinsic volumes of the Gaussian polytope:
\begin{equation*}
\E V_m (\cP_{n,d})
=  (m+1) \binom{n}{m+1} \binom{d}{m} \frac{\kappa_d}{\kappa_{d-m}}  \int_{-\infty}^{\infty} \Phi^{n-m-1}(t) \varphi^{m+1}(t)\dd t.
\end{equation*}

\begin{remark}
Efron~\cite[p.~343]{efron} states that in the Gaussian simplex case $n=d+1$  the right-hand side of~\eqref{eq:V_k_S_n_tsirelson} should be doubled. This does not seem to be necessary. Indeed, Miles~\cite{miles} computed the moments of $\Vol_d (\cP_{d+1,d})$. Formula~\eqref{eq:V_k_S_n_tsirelson} for $n=d+1$ and the formula of Miles read as follows:
$$
\E \Vol_d (\cP_{d+1,d}) = \frac{\sqrt{d+1}}{\Gamma(\frac d2 + 1) 2^{d/2}},
\quad
\E \Vol_d (\cP_{d+1,d}) = \frac{\sqrt{d+1}\, 2^{d/2}\Gamma(\frac d2+ \frac 12)}{\Gamma(\frac 12)d!}.
$$
These formulae are equivalent by the Legendre duplication formula for the Gamma function.
\end{remark}

\subsection{Example: Regular crosspolytope and the symmetric Gaussian polytope}\label{subsec:cross_tsirelson}
Let us now apply Theorem~\ref{theo:tsirelson_spectrum} to the regular crosspolytope
$$
T:= C^n = \conv[\pm e_1,\ldots, \pm e_n]\subset \R^n.
$$
The basic observation is that the  $d$-dimensional spectrum $\spec_d C^n$ is the convex hull of $n$ independent standard Gaussian vectors $X_1,\ldots,X_n$ in $\R^d$ along with their opposites $-X_1,\ldots,-X_n$. That is to say, $\spec_d C^n$ has the same distribution as the symmetric Gaussian polytope $\cP_{n,d}^\pm$. Theorem~\ref{theo:tsirelson_spectrum} yields
\begin{equation}\label{eq:tsirelson_crosspoly}
V_d(C^{n})= \frac{(2\pi)^{d/2}}{d!\kappa_d} \E\,\Vol_d(\cP_{n,d}^\pm), \quad  1\leq d\leq n.
\end{equation}
To compute $V_d(C^n)$ for $1\leq d < n$ observe that all $d$-faces of $C^n$ are isometric to $S^{d}$, their number is $2^{d+1} \binom {n}{d+1}$, and the $d$-volume of each face is $\sqrt{d+1}/d!$. By~\eqref{eq:intrinsic_volume_external_angles} and \eqref{eq:ext_angle_cross_betke} one obtains that
\begin{equation}\label{eq:V_k_C_n}
V_d(C^n) = (2\pi)^{d/2}  \binom{n}{d+1} \frac{d+1}{d!} \int_{0}^{\infty} (2\Phi(t)-1)^{n-d-1} (2\varphi(t))^{d+1}\dd t,
\quad 0\leq d < n,
\end{equation}
which is due to~\citet[Theorem~2.1]{betke_henk}.
Combining Theorem~\ref{theo:asympt_E_Vol} with~\eqref{eq:tsirelson_crosspoly}, we obtain the asymptotic formula
$$
V_d(C^{n}) = \frac {(2\pi)^{d/2}}{d!} \left(u_{2n}^d + d u_{2n}^{d-2} \left(\gamma - \sum_{j=2}^d \frac 1j\right) + o(u_{2n}^{d-2})\right)\sim \frac{(4\pi \log n)^{d/2}}{d!},
$$
where $u_{2n}$ is given by~\eqref{eq:def_u_n_0}, $d$ stays fixed, and $n\to\infty$.
By~\eqref{eq:tsirelson_crosspoly}, formula~\eqref{eq:V_k_C_n} is equivalent to
\begin{equation}\label{eq:V_k_C_n_tsirelson}
\E \Vol_d(\cP_{n,d}^\pm) = \frac{\kappa_d \, n!}{d!(n-d-1)!}\int_{0}^{\infty} (2\Phi(t)-1)^{n-d-1} (2\varphi(t))^{d+1}\dd t,
\;\;\; 1\leq d < n,
\end{equation}
which was stated in Theorem~\ref{theo:exp_polytopes}.
For $d=n$, we have $V_d(C^d) = 2^d/d!$ and hence~\eqref{eq:tsirelson_crosspoly} yields
$$
\E\,\Vol_d(\cP_{d,d}^\pm) = \frac{2^{d/2}}{\Gamma(\frac d2 +1)}.
$$
Using Kubota's formula~\eqref{eq:kubota} and the fact that the $m$-dimensional projection of $\cP_{n,d}^\pm$ has the same law as $\cP_{n,m}^\pm$, we obtain arbitrary expected intrinsic volumes of the symmetric Gaussian polytope:
$$
\E V_m(\cP_{n,d}^\pm) = (m+1) \binom{n}{m+1} \binom{d}{m} \frac{\kappa_d}{\kappa_{d-m}} \int_{0}^{\infty} (2\Phi(t)-1)^{n-m-1} (2\varphi(t))^{m+1}\dd t
$$
for $1\leq d < n$. \citet{HMR04} derived the asymptotics of the expected number of faces of $\cP_{n,d}, \cP_{n,d}^\pm$, as well as of  $\E V_{d-1}(\cP_{n,d})$ and $\E V_{d-1}(\cP_{n,d}^\pm)$ (which is half the surface area of the respective polytope).

\subsection{Example: Rectangular simplex and the Gaussian polytope with \texorpdfstring{$0$}{0}}\label{subsec:simplex_tsirelson_rectangular}
The $d$-dimensional spectrum of the rectangular simplex
$$
T := S^{n}_0 = \conv[0,e_1,\ldots,e_n] \subset \R^n
$$
can be identified with the Gaussian polytope with $0$: $\cP_{n,d}^{(0)} = \conv [0, X_1,\ldots,X_n]$.
Tsirelson's Theorem~\ref{theo:tsirelson_spectrum} yields  the relation
\begin{equation}\label{eq:tsirelson_simplex_zero}
V_d(S^{n}_0)= \frac{(2\pi)^{d/2}}{d!\kappa_d} \E\,\Vol_d(\cP_{n,d}^{(0)}), \quad  1\leq d \leq  n.
\end{equation}
Betke and Henk~\cite[Theorem~2.2]{betke_henk} computed the intrinsic volumes of $S^n_0$:
\begin{equation}\label{eq:V_k_S_n_zero}
V_d(S^{n}_0) = \frac{\binom n d}{d! 2^{n-d}} + (2\pi)^{d/2} \binom{n}{d+1}\frac{d+1}{d!}  \int_0^{\infty}  \Phi^{n-d-1}(t)\varphi^{d+1}(t)\dd t
\end{equation}
for $1\leq d\leq n$.
Together with the relation~\eqref{eq:tsirelson_simplex_zero} this yields formula~\eqref{eq:exp_vol_P_n_d_zero} for $\E \Vol_d(\cP_{n,d}^{(0)})$. See  Section~\ref{sec:non_regular} for more general rectangular simplices.

The asymptotics of $V_d(S^{n}_0)$ and $\E \Vol_d(\cP_{n,d}^{(0)})$ as $n\to\infty$ are exactly the same as for $V_d(S^{n-1})$ and $\E \Vol_d(\cP_{n,d})$ since the term $\binom n d/(d! 2^{n-d})$ is exponentially small, while the integral in~\eqref{eq:V_k_S_n_zero} has the same asymptotics as the integral in~\eqref{eq:V_k_S_n} by Remark~\ref{rem:same_asymptotics}, below.


\subsection{Poissonized versions}\label{subsec:poissonized}
Let $N(\lambda)$ be a Poisson random variable with parameter $\lambda>0$ and assume that $N(\lambda)$ is independent of everything else. Then, the expected volumes of the ``poissonized'' Gaussian polytopes generated by $N(\lambda)$ points are given by
\begin{align}
\E \Vol_d(\cP_{N(\lambda),d})
&=
\kappa_d  \frac{\lambda^{d+1}} {d!}\int_{-\infty}^{\infty} \eee^{\lambda (\Phi(t)-1)} \varphi^{d+1}(t)\dd t, \label{eq:exp_vol_P_n_d_poi}\\
\E \Vol_d(\cP_{N(\lambda),d}^\pm)
&=
\kappa_d  \frac{(2\lambda)^{d+1}} {d!} \int_{0}^{\infty}\eee^{2\lambda(\Phi(t)-1)} \varphi^{d+1}(t)\dd t.
\label{eq:exp_vol_P_n_d_symm_poi}
\end{align}
For example, the first formula can be derived as follows: Using the formula for the total probability and then~\eqref{eq:exp_vol_P_n_d}, we obtain
\begin{multline}
\E \Vol_d(\cP_{N(\lambda),d})
=
\sum_{n=d+1}^\infty  \eee^{-\lambda} \frac{\lambda^{n}}{n!}\E \Vol_d(\cP_{n,d})
\\=
\kappa_d \eee^{-\lambda} \frac{\lambda^{d+1}} {d!} \sum_{n=d+1}^\infty \frac {\lambda^{n-d-1}}{(n-d-1)!} \int_{-\infty}^{\infty} \Phi^{n-d-1}(t) \varphi^{d+1}(t)\dd t.
\end{multline}
Interchanging the sum and the integral and using the Taylor series of the exponential yields~\eqref{eq:exp_vol_P_n_d_poi}. The proof of~\eqref{eq:exp_vol_P_n_d_symm_poi} is similar. The asymptotic behavior of the poissonized polytopes is essentially the same as for the usual ones, see Theorem~\ref{theo:asympt_E_Vol}, but we have to replace $n$ by $\lambda$ throughout:
\begin{align}
\E \Vol_d(\cP_{N(\lambda),d})
&=
\kappa_d  u_\lambda^d + d\kappa_d  u_\lambda^{d-2} \left(\gamma - \sum_{j=2}^d \frac 1j\right) + o(u_\lambda^{d-2}),\label{eq:E_Vol_asympt_simplex_poi}\\
\E \Vol_d(\cP_{N(\lambda),d}^\pm)
&=
\kappa_d  u_{2\lambda}^d + d\kappa_d  u_{2\lambda}^{d-2} \left(\gamma - \sum_{j=2}^d \frac 1j\right) + o(u_{2\lambda}^{d-2}), \label{eq:E_Vol_asympt_cross_poi}
\end{align}
as $\lambda\to\infty$, where $u_\lambda$ is defined by~\eqref{eq:def_u_n_0} with $n$ replaced by $\lambda$. This will be shown in Theorem~\ref{theo:asympt_integral_poi}.

\subsection{Example: Cube and the Gaussian zonotope}\label{subsec:cube_tsirelson}
The intrinsic volumes of the unit cube $Q_n=[0,1]^n$ are well known: $V_k(Q_n) = \binom nk$. It is easy to see that the spectrum $\spec_d Q_n$ is a Gaussian zonotope defined as the Minkowski sum of the segments $[0,X_1], \ldots, [0,X_n]$, where $X_1,\ldots,X_n$ are independent standard Gaussian vectors in $\R^d$:
$$
\zon\{X_1,\ldots,X_n\} = \{ t_1 X_1 + \ldots + t_nX_n \colon t_1,\ldots,t_n\in [0,1]\}.
$$
Using Theorem~\ref{theo:tsirelson_spectrum} one obtains a formula for the expected volume of the Gaussian zonotope:
$$
\E \Vol_d(\zon\{X_1,\ldots,X_n\}) = \frac{n!}{2^{d/2} (n-d)! \Gamma\left(\frac d2 + 1\right)}.
$$
Using Kubota's formula, one obtains a formula for all expected intrinsic volumes of the Gaussian zonotope:
$$
\E V_m(\zon\{X_1,\ldots,X_n\}) = \binom dm \frac{\kappa_d}{\kappa_m\kappa_{d-m}}\cdot \frac{n!}{2^{m/2} (n-m)! \Gamma\left(\frac m2 + 1\right)}.
$$
Somewhat more generally, one can consider the parallelotope $Q_n(l_1,\ldots,l_n) := [0,l_1]\times \ldots \times [0, l_n]\subset \R^n$ with side lengths $l_1,\ldots,l_n>0$. The intrinsic volumes are given by the elementary symmetric functions of $l_1,\ldots,l_n$:
\begin{equation}\label{eq:symm_poly}
V_k([0,l_1]\times \ldots \times [0, l_n])  = \sum_{1\leq i_1 < \ldots < i_k \leq n} l_{i_1}\ldots l_{i_k}:= s_k(l_1,\ldots,l_n).
\end{equation}
The $k$-dimensional spectrum  is the Minkowski sum of the segments $[0,l_1X_1], \ldots, [0, l_n X_n]$, and we obtain from Tsirelson's Theorem~\ref{theo:tsirelson_spectrum} that
$$
\E \Vol_d(\zon\{l_1X_1,\ldots,l_nX_n\}) = \frac{d!s_d(l_1,\ldots,l_n)}{2^{d/2} \Gamma\left(\frac d2 + 1\right)}.
$$

\begin{remark}
Further applications of Tsirelson's theorem can be obtained by taking $T$ to be the Schl\"afli orthoscheme $\conv[e_1,e_1+e_2,\ldots, e_1+\ldots+e_n]$, which allows to compute the expected volume of the convex hull generated by a Gaussian random walk; see~\cite{kabluchko_zaporozhets_sobolev}, where also the zonotope generated by the Gaussian random walk and some further examples were studied.
\end{remark}

\section{Heteroscedastic Gaussian polytopes}\label{sec:non_regular}
Recall that $X_1,\ldots,X_n$ is a standard normal sample in $\R^d$. In this section we shall compute the expected volumes of the polytopes
\begin{align*}
&\cP_{n,d} (l_1,\ldots,l_n) := \conv [l_1 X_1,\ldots,l_n X_n],\\
&\cP_{n,d}^\pm  (l_1^+,l_1^-,\ldots,l_n^+, l_n^-) :=  \conv[l_1^+ X_1, -l_1^- X_1, \ldots,l_n^+ X_n, -l_n^- X_n],
\end{align*}
where $l_i, l_i^+, l_i^-$, $1\leq i\leq n$,  are positive parameters. In the case when all parameters are equal to $1$, we recover the random polytopes $\cP_{n,d}$ and $\cP_{n,d}^\pm$. By Tsirelson's Theorem~\ref{theo:tsirelson_spectrum}, it suffices to compute the intrinsic volumes of the corresponding non-regular simplices and crosspolytopes. This will  generalize a formula due to~\citet{henk_cifre,henk_cifre_notes}.

\subsection{Non-regular simplices}
Fix some $l_1,\ldots,l_n > 0$ and consider the simplex
\begin{equation}\label{eq:def_S}
S= \conv [l_1 e_1,\ldots, l_n e_n],
\end{equation}
where $e_1,\ldots,e_n$ is the standard orthonormal basis in $\R^n$.
Recall that $\varphi$ and $\Phi$ are the density and the distribution function of the standard normal distribution; see~\eqref{eq:def_phi}.
\begin{theorem}\label{theo:intrinsic_volumes_simplex_gen}
For every $1\leq m\leq n$, the $(m-1)$-st intrinsic volume of the simplex $S$ is given by the formula
\begin{multline*}
V_{m-1} (S)
=
\frac{1}{(m-1)!}
\sum_{1\leq i_1<\ldots < i_m \leq n}
\Bigg\{l_{i_1}\ldots l_{i_m} \left(\frac1{l_{i_1}^2} + \ldots + \frac1{l_{i_m}^2}\right)
\phantom{\left(x\sqrt{\frac1{l_{i_1}^2} + \ldots + \frac1{l_{i_m}^2}}\right)}\\
\int_{\R}  \varphi\left(x\sqrt{\frac1{l_{i_1}^2} + \ldots + \frac1{l_{i_m}^2}}\right) \prod_{i\notin \{i_1,\ldots,i_m\}} \Phi\left(\frac x {l_i}\right)\dd x\Bigg\}.
\end{multline*}
\end{theorem}
\begin{remark}
The $d$-dimensional spectrum of $S$ is the heteroscedastic Gaussian polytope $\conv[l_1X_1,\ldots,l_nX_n]$.  Tsirelson's Theorem~\ref{theo:tsirelson_spectrum} yields
\begin{equation}\label{eq:tsirelson_heterosced_gauss_poly}
\E \Vol_d (\conv [l_1X_1,\ldots,l_n X_n]) = \frac{d!  V_d(S)}{\Gamma(\frac d2+1) 2^{d/2}}.
\end{equation}
\end{remark}
\begin{example}\label{ex:simplex_non_reg_with_zero}
Let us fix $l_1,\ldots,l_{n-1}>0$ and let $l_n\downarrow 0$. It follows from Theorem~\ref{theo:intrinsic_volumes_simplex_gen} that the intrinsic volumes of the simplex $S:= \conv [0, l_1e_1,\ldots, l_{n-1} e_{n-1}]$ are given by
\begin{multline*}
(m-1)! V_{m-1} (S) = \frac 1 {2^{n-m}}\sum_{1\leq i_1 < \ldots < i_{m-1} < n}  l_{i_1} \ldots l_{i_{m-1}}
\\+ \sum_{1\leq i_1 < \ldots < i_m < n}
l_{i_1}\ldots l_{i_m} \left(\frac1{l_{i_1}^2} + \ldots + \frac1{l_{i_m}^2}\right) \int_0^{\infty} \varphi \left(x\sqrt{\frac1{l_{i_1}^2} + \ldots + \frac1{l_{i_m}^2}}\right) \prod_{\substack{1\leq i < n\\i\notin \{i_1,\ldots,i_m\} }} \Phi\left(\frac x {l_i}\right)\dd x.
\end{multline*}
For the expected volume of the Gaussian polytope with $0$, we obtain, by taking $l_1=\ldots=l_{n-1} = 1$ and using~\eqref{eq:tsirelson_heterosced_gauss_poly},
\begin{multline*}
 \E \Vol_d (\conv [0, X_1,\ldots,X_{n-1}])
\\
\begin{aligned}
&=
\frac 1 {\Gamma(\frac d2+1) 2^{d/2}} \left\{\frac 1 {2^{n-d-1}} \binom {n-1} d + \binom {n-1} {d+1} (d+1) \int_0^{\infty} \varphi(x\sqrt{d+1}) \Phi^{n-d-2}(x)\dd x\right\}\\
&=
\frac {\binom {n-1}d} {2^{n- \frac d2-1} \Gamma(\frac d2+1) } + \frac{\kappa_d (n-1)!}{d!(n-d-2)!} \int_0^{\infty}  \Phi^{n-d-2}(x)\varphi^{d+1}(x)\dd x,
\end{aligned}
\end{multline*}
which proves~\eqref{eq:exp_vol_P_n_d_zero} if we insert $n+1$ instead of $n$.
\end{example}
\begin{remark}
Recall that an $(n-1)$-dimensional simplex is determined (up to isometry) by $n(n-1)/2$ parameters, for example the edge lengths. The family of simplices considered here is determined by $n$ parameters. For $n>3$, it is a strict subfamily of the family of all simplices.
\end{remark}

\begin{lemma}\label{lem:volume_simplex}
The $(n-1)$-dimensional volume of the simplex $S$ is given by
$$
\Vol_{n-1}(S) =
\frac 1{(n-1)!} \sqrt{s_{n-1}(l_1^2,\ldots,l_n^2)}
=
\frac {l_1\ldots l_n}{(n-1)!} \sqrt{\frac{1}{l_1^2} + \ldots +\frac{1}{l_n^2}}.
$$
\end{lemma}
\begin{proof}
The lemma is known, see~\cite[p.~737]{henk_cifre}, but we give a proof for completeness.
The simplex $S$ is the convex hull of the vectors $l_1e_1,\ldots,l_ne_n$. Its volume is equal to the volume of the parallelotope spanned by the $n-1$ vectors
\begin{equation}\label{eq:vectors_spanning_parallelotope}
l_2e_2-l_1e_1, l_3e_3-l_1e_1,\ldots, l_ne_n-l_1e_1
\end{equation}
divided by $(n-1)!$. That is,
\begin{equation}\label{eq:Vol_n-1_S}
\Vol_{n-1}(S)
=\frac 1 {(n-1)!} \sqrt{\det G},
\end{equation}
where $G$ is the $(n-1)\times (n-1)$ Gram matrix of the vectors~\eqref{eq:vectors_spanning_parallelotope}:
$$
G=
\begin{pmatrix}
l_1^2+l_2^2 &l_1^2      &\ldots   &l_1^2\\
l_1^2   &l_1^2+l_3^2    &\ldots   &l_1^2\\
\vdots &\vdots &\vdots &\vdots  \\
l_1^2   &l_1^2  &\ldots   &l_1^2 + l_n^2
\end{pmatrix}.
$$
It is an exercise to check that
$
\det G
= l_1^2\ldots l_n^2 \sum_{i=1}^n \frac 1{l_i^2}.
$
Inserting this into~\eqref{eq:Vol_n-1_S}, we obtain the statement of the lemma.
\end{proof}

\begin{proof}[Proof of Theorem~\ref{theo:intrinsic_volumes_simplex_gen}.]
The $(m-1)$-st intrinsic volume of $S$ is given by
\begin{equation}\label{eq:intr_volume_faces_angles}
V_{m-1}(S) = \sum_{F\in \cF_{m-1}(S)} \Vol_{m-1}(F) \gamma(F, S),
\end{equation}
where the sum is taken over all $(m-1)$-dimensional faces $F$ of $S$, and $\gamma(F, S)$ denotes the external angle at the face $F$; see Section~\ref{subsec:intrinsic_and external}.

\vspace*{2mm}
\noindent
\textsc{Step 1.}
Let us first identify the faces of $S$ and compute their volumes. Take a vector $a=(a_1,\ldots,a_n)\in \R^n$. Consider the support function
$$
M(a) := \sup_{x\in S} \langle a,x \rangle
= \max_{i=1,\ldots,n} a_il_i,
$$
where the last step uses the definition of $S$; see~\eqref{eq:def_S}. The face of $S$ in the direction of the vector $a$ is the set
$$
F(a)=\{x\in S\colon \langle a, x \rangle =M(a)\}.
$$
Denote by $J=J(a)$ the set of indices $j$ for which $a_jl_j$ becomes maximal:
$$
J=\{1\leq j\leq n\colon a_jl_j =M(a)\}.
$$
To simplify the notation, we can apply a permutation of $\{1,\ldots,n\}$ so that $J=\{1,\ldots,m\}$ for some $1\leq m\leq n$. So, the vector $a$ satisfies
$$
M(a) = a_1l_1=\ldots = a_ml_m > \max_{m+1 \leq i \leq n} a_il_i.
$$
Then, the face $F(a)$ has the form
\begin{equation}\label{eq:F_a}
F(a) = \conv[l_1 e_1,\ldots, l_me_m].
\end{equation}
So, $F(a)$ is an $(m-1)$-dimensional simplex of the same form as $S$, but with $n$ replaced by $m$. By Lemma~\ref{lem:volume_simplex}, we have
\begin{equation}\label{eq:volume_face_simplex}
\Vol_{m-1}(F(a)) = \frac{1}{(m-1)!} l_1\ldots l_m \sqrt{\frac1{l_1^2} + \ldots + \frac1{l_m^2}}.
\end{equation}

\vspace*{2mm}
\noindent
\textsc{Step 2.} Let us now compute the external angles of the simplex $S$.
The external cone attached to the face $F=F(a)$ (which, without restriction of generality,  is taken to be of the form~\eqref{eq:F_a}) is  the set
$$
\Nor (S, F) = \{b\in \R^n\colon F(b) \supset F\}.
$$
Equivalently,
$$
\Nor (S,F) = \left\{b=(b_1,\ldots,b_n)\in \R^n\colon b_1l_1 = \ldots = b_ml_m \geq  \max_{m+1 \leq i \leq n} b_il_i\right\}.
$$
Denote by $L(F)$ the linear subspace spanned by $\Nor(S, F)$:
$$
L(F) = \left\{b=(b_1,\ldots,b_n)\in \R^n\colon b_1l_1 = \ldots = b_ml_m\right\}.
$$
Let $\mu$ be the standard $n$-dimensional Gaussian measure on $\R^n$. Further, let $\tilde \mu$ be a standard $(n-m+1)$-dimensional  Gaussian measure on $L(F)$ characterized by
\begin{equation}\label{eq:char_1}
\int_{L(F)} \eee^{i \langle t,b\rangle} \tilde \mu(\dd b) = \eee^{-\frac 12 \langle t,t\rangle},
\quad
t\in L(F).
\end{equation}
We will consider $\R^n$ as a probability space endowed with measure $\mu$ or $\tilde \mu$. Consider the random variables  $\xi_1,\ldots,\xi_n:\R^n\to\R$ and the random vector $\xi=(\xi_1,\ldots,\xi_n)$ defined by
$$
\xi_i(b_1,\ldots,b_n) =b_i, \quad 1\leq i\leq n.
$$
Under the measure $\mu$, the random variables $\xi_1,\ldots,\xi_n$ are independent and standard normal. Let us compute the distribution of $\xi_1,\ldots,\xi_n$ under $\tilde \mu$. Since for any vector $t=(t_1,\ldots,t_n)\in L(F)$ we can write $s:=t_1l_1=\ldots=t_ml_m$, we have from~\eqref{eq:char_1} that
\begin{equation}\label{eq:E_tilde_mu_1}
\E_{\tilde \mu} \eee^{i \langle t, \xi\rangle}
=
\eee^{-\frac 12 \left(s^2 \sigma_*^2 + t_{m+1}^2 + \ldots + t_n^2 \right)},
\end{equation}
where we used the notation
\begin{equation}\label{eq:sigma_*_def}
\sigma_*^2 = \frac 1 {l_1^2} + \ldots +\frac 1 {l_m^2}.
\end{equation}
Also, under $\tilde \mu$ we can write $\xi_*:= \xi_1l_1=\ldots=\xi_ml_m$ a.s., so that we can rewrite~\eqref{eq:E_tilde_mu_1} as
$$
\E_{\tilde \mu} \eee^{i \left(s\sigma_*^2 \xi_* + t_{m+1} \xi_{m+1} + \ldots + t_{n} \xi_{n}\right)}
=
\eee^{-\frac 12 \left(s^2 \sigma_*^2 + t_{m+1}^2 + \ldots + t_n^2 \right)}.
$$
It follows that under $\tilde \mu$, the random variables $\xi_*$, $\xi_{m+1},\ldots,\xi_n$ are independent and
\begin{equation}\label{eq:xi_*_xi_joint}
\xi_*=\xi_1l_1=\ldots=\xi_ml_m \sim N(0,\sigma_*^{-2}), \quad  \xi_{m+1},\ldots,\xi_n\sim N(0,1).
\end{equation}
Now, the angle of the external cone at face $F=F(a)$ is by definition
$$
\gamma(F, S) = \tilde \mu (\Nor (S,F))
=
\P_{\tilde \mu}[\xi \in \Nor (S,F))]
=
\P_{\tilde \mu}\left[\xi_* \geq  \max_{m+1\leq i \leq n} \xi_i l_i\right].
$$
The probability on the right-hand side can be computed using the joint distribution of $\xi_*,\xi_{m+1},\ldots,\xi_n$ known from~\eqref{eq:xi_*_xi_joint}:
\begin{equation}\label{eq:external_angle_simplex}
\gamma(F,S)= \P_{\tilde \mu}\left[\xi_* \geq  \max_{m+1\leq i \leq n} \xi_i l_i\right]
=
\int_{\R} \sigma_* \varphi\left(x\sigma_*\right) \prod_{i=m+1}^n  \Phi\left(\frac x {l_i}\right)\dd x.
\end{equation}

\vspace*{2mm}
\noindent
\textsc{Step 3.}
Inserting~\eqref{eq:volume_face_simplex}, \eqref{eq:external_angle_simplex}, \eqref{eq:sigma_*_def} into the formula~\eqref{eq:intr_volume_faces_angles}, we obtain the statement of the theorem.
\end{proof}

\subsection{Non-regular crosspolytopes}
In this section we compute the intrinsic volumes of certain non-regular crosspolytopes.
Fix some $l_1^+,\ldots,l_n^+>0$ and $l_1^-,\ldots,l_n^->0$.  Let us agree  to write
$$
\sgn y=
\begin{cases}
+1, \text{ if } y>0,\\
-1, \text{ if } y<0,
\end{cases}
\quad
l_i^{\sgn y}=
\begin{cases}
l_i^+, \text{ if } y>0,\\
l_i^-, \text{ if } y<0.
\end{cases}
$$
Consider the ``generalized crosspolytope''
\begin{equation}\label{eq:def_C}
C=\left\{x=(x_1,\ldots,x_n)\in \R^n \colon \sum_{i=1}^n \frac{|x_i|}{l_i^{\sgn x_i}} \leq  1\right\}.
\end{equation}
Note that in the case $x_i=0$ we interpret $|x_i|/l_i^{\sgn x_i}$ as $0$ even though $l_i^{\sgn x_i}$ is not defined. Denoting by  $e_1,\ldots,e_n$ the standard orthonormal basis in $\R^n$, we see that $C$ is the convex hull of the vectors
$$
l_1^+ e_1,\ldots, l_n^+ e_n, -l_1^- e_1, \ldots, -l_n^- e_n.
$$
\begin{theorem}\label{theo:intrinsic_volumes_crosspoly_gen}
For every $1\leq m\leq n$, the $(m-1)$-st intrinsic volume of the generalized crosspolytope $C$ is given by the formula
\begin{multline*}
V_{m-1} (C)
=
\frac{1}{(m-1)!}
\sum_{\substack{1\leq i_1<\ldots < i_m \leq n \\ (\eps_1,\ldots,\eps_m)\in \{+1,-1\}^m}}
\Bigg\{l_{i_1}^{\eps_1}\ldots l_{i_m}^{\eps_m} \left(\frac1{(l_{i_1}^{\eps_1})^2} + \ldots + \frac1{(l_{i_m}^{\eps_m})^2}\right)\\
\int_{0}^{\infty}  \varphi\left(x\sqrt{\frac1{(l_{i_1}^{\eps_1})^2} + \ldots + \frac1{(l_{i_m}^{\eps_m})^2}}\right) \prod_{i\notin \{i_1,\ldots,i_m\}} \left(\Phi\left(\frac x {l_i^{+}}\right) - \Phi\left(-\frac x {l_i^{-}} \right)\right)\dd x\Bigg\}.
\end{multline*}
\end{theorem}
\begin{remark}
By Tsirelson's Theorem~\ref{theo:tsirelson_spectrum}, we can compute the volume of the heteroscedastic generalization of the symmetric Gaussian polytope $\cP_{n,d}^{\pm}$:
$$
\E \Vol_d (\conv [l_1^+X_1,-l_1^- X_1, \ldots,l_n^+ X_n, -l_n^- X_n]) = \frac{d!  V_d(C)}{\Gamma(\frac d2+1) 2^{d/2}}.
$$
\end{remark}
Let us mention some special cases of Theorem~\ref{theo:intrinsic_volumes_crosspoly_gen}.
\begin{example}
Let $l_i:=l_i^+=l_i^-$ for all $1\leq i\leq d$, which means that the generalized crosspolytope is symmetric with respect to the origin.  Then, Theorem~\ref{theo:intrinsic_volumes_crosspoly_gen} simplifies as follows:
\begin{multline*}
V_{m-1} (C)
=
\frac{2^m}{(m-1)!}
\sum_{1\leq i_1<\ldots < i_m \leq n}
\Bigg\{l_{i_1}\ldots l_{i_m} \left(\frac1{l_{i_1}^2} + \ldots + \frac1{l_{i_m}^2}\right)\\
\int_{0}^{\infty}  \varphi\left(x\sqrt{\frac1{l_{i_1}^2} + \ldots + \frac1{l_{i_m}^2}}\right)
\prod_{i\notin \{i_1,\ldots,i_m\}} \left(2\Phi\left(\frac x {l_i}\right) - 1\right)\dd x\Bigg\},
\end{multline*}
which recovers a result of \citet[Corollary~2.1]{henk_cifre}.
\end{example}

\begin{proof}[Proof of Theorem~\ref{theo:intrinsic_volumes_crosspoly_gen}.]
As in Theorem~\ref{theo:intrinsic_volumes_simplex_gen}, the proof is again based on the formula for the $(m-1)$-st intrinsic volume of $C$:
\begin{equation}\label{eq:intr_volume_faces_angles_cross}
V_{m-1}(C) = \sum \Vol_{m-1}(F) \gamma(F, C),
\end{equation}
where the sum is taken over all $(m-1)$-dimensional faces $F$ of $C$.

\vspace*{2mm}
\noindent
\textsc{Step 1.}
We start by identifying the faces of $C$ and computing their volumes. Take a vector $a=(a_1,\ldots,a_n)\in \R^n\bsl\{0\}$.
Consider the support function
$$
M(a)
:= \sup_{x\in C} \langle a,x \rangle
= \sup_{x\in C} \sideset{}{'} \sum_{i=1}^n |a_i| l_i^{\sgn a_i} \, \frac {x_i \sgn a_i}{l_i^{\sgn a_i}}
= \max_{i=1,\ldots,n} |a_i|l_i^{\sgn a_i},
$$
where $\sum'$ indicates that the sum is restricted to $i$ for which $a_i\neq 0$ and  the last equality uses~\eqref{eq:def_C} together with
$$
\frac {x_i \sgn a_i}{l_i^{\sgn a_i}} \leq
\frac {|x_i|}{l_i^{\sgn x_i}}.
$$
The face of $C$ in the direction of the vector $a$ is the set
$$
F(a)=\{x\in C\colon \langle a, x \rangle = M(a)\}.
$$
After renumbering the $a_i$'s we can assume that for some $1\leq m\leq n$,
$$
M(a) = |a_1|l_1^{\sgn a_1}=\ldots = |a_m| l_m^{\sgn a_m} > \max_{m+1 \leq i \leq n} |a_i| l_i^{\sgn a_i}.
$$
Also, $a_1,\ldots,a_m$ are non-zero. Then, the face $F(a)$ has the form
\begin{equation}\label{eq:F_a_cross}
F(a) = \left\{x=(x_1,\ldots,x_m,0,\ldots,0) \in \R^n \colon \sum_{i=1}^m \frac{|x_i|}{l_i^{\sgn x_i}}=1, x_i\in \R_{\sgn a_i}, 1\leq i\leq m \right\}.
\end{equation}
Here, $\R_{\sgn a_i}$ is equal to $[0,\infty)$ or $(-\infty, 0]$ depending on the sign of $a_i\neq 0$. So, $F(a)$ is isometric to an $(m-1)$-dimensional simplex of the same form as the simplex  $S$, but with $n$ replaced by $m$. Let $\eps_i=\sgn a_i\in \{+1,-1\}$, $1\leq i\leq m$. By Lemma~\ref{lem:volume_simplex}, we have
\begin{equation}\label{eq:volume_face_cross}
\Vol_{m-1}(F(a)) = \frac{1}{(m-1)!} l_1^{\eps_1}\ldots l_m^{\eps_m} \sqrt{\frac1{(l_1^{\eps_1})^2} + \ldots + \frac1{(l_m^{\eps_m})^2}}.
\end{equation}

\vspace*{2mm}
\noindent
\textsc{Step 2.} Let us  compute the external angles of the generalized crosspolytope $C$.
Consider a face $F=F(a)$ which we can take to be of the form~\eqref{eq:F_a}.  The external cone attached to $F$ is the set
$$
\Nor (C, F) = \{b\in \R^n\colon F(b) \supset F\}.
$$
Equivalently,
\begin{multline*}
\Nor (C,F) = \left\{b=(b_1,\ldots,b_n)\in \R^n\colon |b_1|l_1^{\eps_1} = \ldots = |b_m| l_m^{\eps_m} \geq  \max_{m+1 \leq i \leq n} |b_i|l_i^{\sgn b_i}, \right.
\\
\left.\sgn b_i=\eps_i, 1\leq i\leq m\right\}.
\end{multline*}
Denote by $L(F)$ the linear subspace spanned by $\Nor(C,F)$:
$$
L(F) = \left\{b=(b_1,\ldots,b_n)\in \R^n\colon l_1^{\eps_1}\eps_1 b_1 = \ldots = l_m^{\eps_m}\eps_m b_m\right\}.
$$
Let $\tilde \mu$ be the standard $(n-m+1)$-dimensional  Gaussian measure on $L(F)$.
Consider the random variables  $\xi_1,\ldots,\xi_n:\R^n\to\R$  defined by
$$
\xi_i(b_1,\ldots,b_n) =b_i, \quad 1\leq i\leq n.
$$
Also, under the probability measure $\tilde \mu$ we can write $\xi_*:= \xi_1l_1^{\eps_1}\eps_1=\ldots=\xi_ml_m^{\eps_m} \eps_m$ a.s.
Arguing exactly as in the proof of Theorem~\ref{theo:intrinsic_volumes_simplex_gen} we obtain that $\tilde \mu$, the random variables $\xi_*$, $\xi_{m+1},\ldots,\xi_n$ are independent and
\begin{equation}\label{eq:xi_*_xi_joint_cross}
\xi_*=\xi_1l_1^{\eps_1}\eps_1=\ldots=\xi_ml_m^{\eps_m}\eps_m \sim N(0,\sigma_*^{-2}), \quad  \xi_{m+1},\ldots,\xi_n\sim N(0,1),
\end{equation}
where
\begin{equation}\label{eq:sigma_*_def_cross}
\sigma_*^2 = \frac 1 {(l_1^{\eps_1})^2} + \ldots + \frac 1 {(l_m^{\eps_m})^2}.
\end{equation}
Now, the angle of the external cone at face $F=F(a)$ is by definition
$$
\gamma(F, C) = \tilde \mu (\Nor(C,F))
=
\P_{\tilde \mu}[\xi \in \Nor(C,F)]
=
\P_{\tilde \mu}\left[\xi_* \geq  \max_{m+1\leq i \leq n} |\xi_i| l_i^{\sgn \xi_i}\right].
$$
The probability on the right-hand side can be computed using the joint distribution of $\xi_*,\xi_{m+1},\ldots,\xi_n$ known from~\eqref{eq:xi_*_xi_joint}:
\begin{equation}\label{eq:external_angle_cross}
\gamma(F, C)= \P_{\tilde \mu}\left[\xi_* \geq  \max_{m+1\leq i \leq n} |\xi_i| l_i^{\sgn \xi_i}\right]
=
\int_{0}^{\infty} \sigma_* \varphi\left(x\sigma_*\right)
\prod_{i=m+1}^n  \left(\Phi\left(\frac x {l_i^{+}}\right) - \Phi\left(-\frac x {l_i^{-}} \right)\right)\dd x.
\end{equation}

\vspace*{2mm}
\noindent
\textsc{Step 3.}
Inserting~\eqref{eq:volume_face_cross}, \eqref{eq:external_angle_cross}, \eqref{eq:sigma_*_def_cross} into the formula~\eqref{eq:intr_volume_faces_angles_cross}, we obtain the statement of the lemma.
\end{proof}

\begin{remark}
There is an alternative proof of Theorem~\ref{theo:intrinsic_volumes_crosspoly_gen} based on the fact that the intrinsic volumes $V_k$ have an additive extension to the convex ring. For every representation of a convex set $T$ as a  union $T_1\cup\ldots\cup T_N$ of (not necessarily disjoint) convex sets, the following inclusion-exclusion principle holds~\cite[Section~14.4]{SW08}:
$$
V_k(T) = \sum_{r=1}^N (-1)^{r-1} \sum_{1\leq j_1 < \ldots < j_r \leq N} V_k(T_{j_1}\cap \ldots \cap T_{j_r}).
$$
We use this formula for $N = 2^n$ and the decomposition
$$
C= \bigcup_{\eps= (\eps_1,\ldots,\eps_n) \in \{-1,+1\}^n} C_{\eps},
$$
where $C_{\eps} =  C_{\eps_1,\ldots,\eps_m} = \conv[0, \eps_1 l_1^{\eps_1} e_1,\ldots, \eps_n l_n^{\eps_n}e_n]$. Now observe that every intersection of the form $C_{\eps^{(1)}} \cap \ldots \cap C_{\eps^{(r)}}$, where $\eps^{(1)},\ldots, \eps^{(r)}\in \{-1,+1\}^n$,  is a simplex of the same form as in Example~\ref{ex:simplex_non_reg_with_zero}. Applying the formula of this example to every term in the inclusion-exclusion formula one can arrive at the result of Theorem~\ref{theo:intrinsic_volumes_crosspoly_gen}.
\end{remark}

\section{Multiple order statistics}\label{sec:multiple}

\subsection{Interpretation of the intrinsic volumes}
Recall that $\xi_1,\ldots,\xi_n$ are independent and standard normal random variables. It immediately follows from Sudakov's Theorem~\ref{theo:sudakov} that
\begin{equation}\label{eq:V_1_as_E_max}
V_1(S^{n-1})=\sqrt{2\pi}\,\E\,\max\{\xi_1,\dots,\xi_n\},
\quad
V_1(C^n)=\sqrt{2\pi}\,\E\,\max \{|\xi_1|,\dots,|\xi_n|\}.
\end{equation}
In~\cite{kabluchko_litvak_zaporozhets}, these relations were exploited to compare both mean widths to each other and to derive their asymptotics; see also~\cite{finch_simplex,finch_cross}. In the following we shall generalize~\eqref{eq:V_1_as_E_max} to higher intrinsic volumes.

\subsubsection*{Regular simplices}
We consider the case of $S^{n-1}$ and indicate the changes necessary for the crosspolytope below.  Let $\xi_{(1)} \leq  \ldots \leq \xi_{(n)}$ be the order statistics of $\xi_1,\ldots,\xi_n$. We shall be interested in the random event
$$
A_{n,k} := \{\xi_{(n)} = \xi_{(n-1)} =\ldots = \xi_{(n-k+1)}\} = \{\xi_{(n)} = \xi_{(n-k+1)}\}.
$$
The event $A_{n,k}$ occurs if the sample maximum has multiplicity $k\in \{1,\ldots,n\}$. For random variables with discrete distributions, such events were studied~\cite{bruss_gruebel}, but in the Gaussian case the probability of $A_{n,k}$ is $0$ (for $k\neq 1$). Thus,  conditioning on this event requires some regularization technique. Therefore, let us consider for small $\eps>0$ the event
\begin{equation}
A_{n,k} (\eps) := \{\xi_{(n)} - \xi_{(n-k+1)}\leq \eps\}.
\end{equation}
\begin{proposition}\label{prop:conditional}
Let $f:\R\to\R$ be Borel function such that $\int_{-\infty}^{\infty} |f(x)| \varphi (x)  \dd x <\infty$. Then, as $\eps\downarrow 0$,
$$
\E[ f(\xi_{(n)}) \ind_{A_{n,k}(\eps)}] \sim \eps^{k-1} k \binom nk \int_{-\infty}^{\infty} f(s) \Phi^{n-k}(s) \varphi^k (s)\dd s,
$$
where we write $a(\eps) \sim b(\eps)$ if $\lim_{\eps\downarrow 0} \frac{a(\eps)}{b(\eps)} = 1$.
\end{proposition}
A rigorous proof will be given in Section~\ref{subsec:proof_conditional}, below. To understand the meaning of the factors on the right-hand side, observe that the event $A_{n,k}(\eps)$ occurs if
\begin{itemize}
\item[(a)] one of the random variables (the maximum) equals some $s\in\R$, contributing $n f(s) \varphi(s)\dd s$;
 \item[(b)] some other $k-1$ random variables take values in the interval $(s-\eps,s)$, contributing $\binom {n-1}{k-1}\varphi^{k-1}(s)\eps^{k-1}$;
 \item[(c)] the remaining $n-k$ random variables are smaller than $s-\eps$, contributing $\Phi^{n-k}(s)$.
\end{itemize}
The next theorem generalizes the first identity in~\eqref{eq:V_1_as_E_max} to arbitrary intrinsic volumes.
\begin{theorem}\label{theo:intrinsic_interpretation_simplex}
For all $1\leq k < n$ the intrinsic volumes of the regular simplex $S^{n-1}$ satisfy
\begin{equation*}
V_k(S^{n-1})
= \frac{(2\pi)^{\frac k2}}{k!} \cdot \lim_{\eps\downarrow 0} \eps^{1-k} \E [\max\{\xi_1,\ldots,\xi_n\} \ind_{A_{n,k}(\eps)}].
\end{equation*}
\end{theorem}
\begin{proof}
By Proposition~\ref{prop:conditional} with $f(s)=s$,
\begin{equation}\label{eq:proof_mult_max_1}
\lim_{\eps\downarrow 0} \eps^{1-k}\E[ \max\{\xi_1,\ldots,\xi_n\} \ind_{A_{n,k}(\eps)}] = k \binom nk \int_{-\infty}^{\infty} s \Phi^{n-k}(s) \varphi^k (s)\dd s.
\end{equation}
On the other hand, recall that $V_k(S^{n-1})$ is given by~\eqref{eq:V_k_S_n}. Using partial integration, we can transform the integral involved in~\eqref{eq:V_k_S_n} as follows:
\begin{multline*}
\int_{-\infty}^{\infty} \Phi^{n-k-1} (s) \varphi^{k+1}(s) \dd s
=
\int_{-\infty}^{\infty} \Phi^{n-k-1} (s) \varphi^{k}(s) \dd \Phi(s)
\\=
-\int_{-\infty}^{\infty}
\Phi(s) \left((n-k-1) \Phi^{n-k-2}(s) \varphi^{k+1}(s) - \Phi^{n-k-1}(s) k \varphi^{k}(s) s  \right) \dd s,
\end{multline*}
where we used the relations $\Phi'(s) = \varphi(s)$ and $\varphi'(s) = -s \varphi(s)$. After some transformations, we arrive at
\begin{equation}\label{eq:proof_mult_max_2}
\int_{-\infty}^{\infty} \Phi^{n-k-1} (s) \varphi^{k+1}(s) \dd s = \frac{k}{n-k} \int_{-\infty}^{\infty} s \Phi^{n-k} (s) \varphi^{k}(s) \dd s.
\end{equation}
The statement of the theorem follows by combining~\eqref{eq:V_k_S_n}, \eqref{eq:proof_mult_max_1} and~\eqref{eq:proof_mult_max_2}.
\end{proof}
\begin{proposition}\label{prop:conditional_limit}
The conditional law of $\xi_{(n)}$ given $A_{n,k}(\eps)$ converges weakly, as $\eps\downarrow 0$, to the probability distribution with the density
\begin{equation}\label{eq:cond_density_mult_max}
f_{\xi_{(n)}} (t| A_{n,k}) :=  \frac{\Phi^{n-k}(t) \varphi^k(t)}{\int_{-\infty}^{\infty} \Phi^{n-k}(s) \varphi^k(s) \dd s}, \quad t\in\R.
\end{equation}
\end{proposition}
\begin{proof}
Taking $f(x)=\ind_{\{x \leq t\}}$ and $f(x) = 1$ in Proposition~\ref{prop:conditional}, we obtain
\begin{align*}
\P[\xi_{(n)} \leq t, A_{n,k}(\eps)] &\sim \eps^{k-1} k \binom nk \int_{-\infty}^{t} \Phi^{n-k}(s) \varphi^k(s) \dd s, \\
\P[A_{n,k}(\eps)] &\sim \eps^{k-1} k \binom nk \int_{-\infty}^{\infty} \Phi^{n-k}(s) \varphi^k(s) \dd s,
\end{align*}
as $\eps\downarrow 0$. Taking the quotient of these relations yields
$$
\lim_{\eps\downarrow 0} \P[\xi_{(n)} \leq t | A_{n,k}(\eps)] = \frac{\int_{-\infty}^{t} \Phi^{n-k}(s) \varphi^k(s) \dd s}{\int_{-\infty}^{\infty} \Phi^{n-k}(s) \varphi^k(s) \dd s},
$$
and the statement follows.
\end{proof}
\begin{remark}
We may consider $f_{\xi_{(n)}} (t| A_{n,k})$ given in~\eqref{eq:cond_density_mult_max} as the density of $\xi_{(n)}$ given the event $A_{n,k}= \{\xi_{(n-k+1)} = \xi_{(n)}\}$ (which has probability $0$ for $k\neq 1$).
\end{remark}

\subsubsection*{Regular crosspolytopes}
This case is very similar to the case of regular simplices, but we need to replace $\xi_1,\ldots,\xi_n$ by $|\xi_1|,\ldots,|\xi_n|$. Let $|\xi|_{(1)} \leq \ldots \leq |\xi|_{(n)}$ be the order statistics of $|\xi_1|,\ldots,|\xi_n|$. Define the random event
$$
B_{n,k}(\eps) := \{|\xi|_{(n)} - |\xi|_{(n-k+1)}\leq \eps\},
$$
which approximates the event $B_{n,k} := \{|\xi|_{(n)}=|\xi|_{(n-k+1)}\}$ whose probability is $0$.
\begin{proposition}\label{prop:conditional_crosspoly}
Let $f: [0,\infty)\to\R$ be Borel function such that $\int_{0}^{\infty} |f(x)| \varphi(x) \dd x <\infty$. Then, as $\eps\downarrow 0$, we have
$$
\E[ f(|\xi|_{(n)}) \ind_{B_{n,k}(\eps)}] \sim \eps^{k-1} k \binom nk \int_{0}^{\infty} f(s) (2\Phi(s)-1)^{n-k} (2\varphi(s))^{k}\dd s.
$$
\end{proposition}
\begin{proposition}
The conditional law of $|\xi|_{(n)}$ given $B_{n,k}(\eps)$ converges weakly, as $\eps\downarrow 0$, to the probability distribution with the density
\begin{equation}\label{eq:cond_density_mult_max_cross}
f_{|\xi|_{(n)}} (t| B_{n,k}) :=  \frac{(2\Phi(t)-1)^{n-k} (2\varphi(t))^k}{\int_{0}^{\infty} (2\Phi(s)-1)^{n-k} (2\varphi(s))^k \dd s}, \quad t\geq 0.
\end{equation}
\end{proposition}
The intrinsic volumes of $C^n$ can be interpreted as follows.
\begin{theorem}
For all $1\leq k < n$ we have
\begin{equation*}
V_k(C^n)
= \frac{(2\pi)^{\frac k2}}{k!} \cdot \lim_{\eps\downarrow 0} \eps^{1-k} \E [\max\{|\xi_1|,\ldots,|\xi_n|\} \ind_{B_{n,k}(\eps)}].
\end{equation*}
\end{theorem}
We omit the proofs of these three results because they are similar to Propositions~\ref{prop:conditional}, \ref{prop:conditional_limit}  and Theorem~\ref{theo:intrinsic_interpretation_simplex}.

\begin{remark}
Similarly to the above, the external angle of the simplex $S^n_0 = \conv[0,e_1,\ldots,e_n]$ at any $k$-face not containing $0$, see~\eqref{eq:ext_angle_simplex_zero}, can be expressed through the expected maximum of $\xi_1,\ldots,\xi_n$ on the event $A_{n,k}^{(0)} := \{\xi_{(n)} = \xi_{(n-k+1)}>0\}$.
\end{remark}

\subsection{Asymptotic distribution of the multiple maximum}
It is well-known in the theory of extreme values, see~\cite[Theorem~1.5.3 on p.~14]{leadbetter_etal_book} or~\cite[Example 2.3.2 on p.~65]{galambos_book}, that
$$
\lim_{n\to\infty} \P\left[\sqrt{2\log n}\,(\max\{\xi_1,\ldots,\xi_n\} - u_n) \leq x\right] = \eee^{-\eee^{-x}}, \quad x\in\R,
$$
where the normalizing sequence $u_n$ is given by
\begin{equation}\label{eq:def_u_n}
u_n = \sqrt{2\log n}  - \frac{\frac 12 \log \log n +\log (2\sqrt{\pi})}{\sqrt{2\log n}}.
\end{equation}


The next proposition describes the limit distribution for the maximum in a normal sample given that the maximum is attained $k$ times.
\begin{proposition}\label{prop:asympt_distr_mult_max}
Let $k\in\N$ be fixed. The law of $u_n (\xi_{(n)} - u_n)$ given $A_{n,k}$ converges weakly, as $n\to\infty$, to the distribution with the density
\begin{equation}\label{eq:density_k_th_order_stat}
\frac 1 {\Gamma(k)} \eee^{-\eee^{-z}} \eee^{-kz}, \quad z\in\R.
\end{equation}
\end{proposition}
\begin{proof}
The density of $\xi_{(n)}$ conditioned on $A_{n,k}$ is given by~\eqref{eq:cond_density_mult_max}.
By Scheff\'e's lemma, it suffices to prove the pointwise convergence of the corresponding densities, that is
\begin{equation}\label{eq:density_k_th_order_stat_need_to_prove}
\lim_{n\to\infty} \frac 1 {u_n} f_{\xi_{(n)}} \left(\left.u_n + \frac z {u_n}\right| A_{n,k}\right) = \frac 1 {\Gamma(k)} \eee^{-\eee^{-z}} \eee^{-kz}.
\end{equation}
It is easy to check that the sequence $u_n$ given by~\eqref{eq:def_u_n} satisfies
$$
1-\Phi(u_n)\simn \frac{1}{\sqrt{2\pi} u_n}\eee^{-u_n^2/2} \simn \frac 1n.
$$
It follows that for all $z\in\R$,
\begin{equation}\label{eq:Phi_asympt}
1-\Phi\left(u_n + \frac z {u_n}\right) \simn \frac{\eee^{-z}}{n},
\quad
\varphi\left(u_n + \frac z {u_n}\right) \simn  \frac{u_n\eee^{-z}}{n}.
\end{equation}
Since $k$ is fixed, we obtain
\begin{equation}\label{eq:asympt_Phi_tech_1}
\Phi^{n-k}\left(u_n + \frac z {u_n}\right) \varphi^k \left(u_n + \frac z {u_n}\right)\simn
\eee^{-\eee^{-z}} \eee^{-kz} \frac{u_n^k}{n^k}.
\end{equation}
Also, writing $s= u_n + \frac{z}{u_n}$ and using~\eqref{eq:asympt_Phi_tech_1}, we obtain
\begin{multline}\label{eq:asympt_Phi_tech_2}
\int_{-\infty}^{\infty} \Phi^{n-k}(s) \varphi^k(s) \dd s = \int_{-\infty}^{\infty} \Phi^{n-k}\left(u_n + \frac z {u_n}\right) \varphi^k\left(u_n + \frac z {u_n}\right) \frac{\dd z}{u_n}
\\\simn  \frac{u_n^{k-1}}{n^k} \int_{-\infty}^{\infty} \eee^{-\eee^{-z}} \eee^{-kz} \dd z
=
\frac{u_n^{k-1}}{n^k} \Gamma(k).
\end{multline}
Taking the quotient of~\eqref{eq:asympt_Phi_tech_1} and~\eqref{eq:asympt_Phi_tech_2}, we arrive at the desired formula~\eqref{eq:density_k_th_order_stat_need_to_prove}.
\end{proof}

Observe that~\eqref{eq:density_k_th_order_stat} is also the limit density of the appropriately normalized $k$-th upper order statistics $\xi_{(n-k+1)}$ (without any conditioning) in the sense that
$$
\lim_{n\to\infty} \P\left[\sqrt{2\log n}\,( \xi_{(n-k+1)} - u_n) \leq x\right] = \frac 1 {\Gamma(k)} \int_{-\infty}^x \eee^{-\eee^{-z}} \eee^{-kz}\dd z= \eee^{-\eee^{-x}}\sum_{j=0}^{k-1}\frac{\eee^{-jx}}{j!}
$$
for all $x\in\R$; see~\cite[Theorem~2.2.2 on p.~33]{leadbetter_etal_book} or~\cite[Example~2.8.1 on p.~105]{galambos_book}.

\begin{remark}
A result similar to Proposition~\ref{prop:asympt_distr_mult_max} holds for the conditional density of $|\xi|_{(n)}$ given $B_{n,k}$, but the centering sequence $u_n$ should be replaced by $u_{2n}$ expressing the fact that $|\xi|_{(n)}$ is a maximum of $2n$ standard Gaussian variables $\pm \xi_1,\ldots,\pm \xi_n$ which are approximately independent.   The limit density remains the same.
\end{remark}

\subsection{Proof of Proposition~\ref{prop:conditional}}\label{subsec:proof_conditional}

The joint density of the order statistics $(\xi_{(1)}, \ldots, \xi_{(n)})$ is $n!\varphi(x_1) \ldots \varphi(x_n) \ind_{\{x_1 < \ldots < x_n\}}$; see~\cite[Eq.~(2.9) on p.~8]{nevzorov_book}, hence we can write
$$
\E [f(\xi_{(n)}) \ind_{A_{n,k}(\eps)}]
=
n! \int_{\substack{x_1<\ldots<x_n\\x_n-x_{n-k+1} < \eps}} f(x_n) \varphi(x_1) \ldots \varphi(x_n) \dd x_1\ldots \dd x_n.
$$
We replace the integration over $x_1,\ldots,x_n$ by the integration over $x_1,\ldots,x_{n-k}, s_1,\ldots,s_{k-1},x_n$, where the new variables $s_1,\ldots, s_{k-1}$ are introduced as follows:
$$
x_{n-1} = x_n - \eps s_1,\quad \ldots, \quad x_{n-k+1} = x_n - \eps s_{k-1}.
$$
Noting that the Jacobian is $\eps^{k-1}$, we can transform the above integral as follows:
\begin{multline}\label{eq:integral_dominated}
\eps^{1-k} \E [f(\xi_{(n)}) \ind_{A_{n,k}(\eps)}]
=
n!  \int_{\R^n} \ind_{\{x_1<\ldots <x_{n-k} < x_n -\eps s_{k-1}\}} \ind_{\{0 < s_1 <\ldots < s_{k-1} <1\}}\cdot
\\ \cdot f(x_n) \varphi(x_n)\dd x_n \, \prod_{i=1}^{n-k} \varphi(x_i) \dd x_i \, \prod_{i=1}^{k-1} \varphi(x_n -\eps s_{i}) \dd s_i.
\end{multline}
By the dominated convergence theorem (whose conditions will be verified below), we have
\begin{multline*}
\lim_{\eps \downarrow 0}\eps^{1-k} \E [f(\xi_{(n)}) \ind_{A_{n,k}(\eps)}]
=
n!  \int_{\R^n}
\ind_{\{x_1<\ldots <x_{n-k} < x_n\}} \ind_{\{0 < s_1 <\ldots < s_{k-1} <1\}} \cdot
\\
\cdot f(x_n) \varphi(x_n)\dd x_n \, \prod_{i=1}^{n-k} \varphi(x_i) \dd x_i \, \prod_{i=1}^{k-1} \varphi(x_n) \dd s_i.
\end{multline*}
Performing the integration over $s_1,\ldots,s_{k-1}$ yields the factor $1/(k-1)!$:
\begin{multline*}
\lim_{\eps \downarrow 0}\eps^{1-k} \E [f(\xi_{(n)}) \ind_{A_{n,k}(\eps)}]
=
\frac{n!}{(k-1)!}  \int_{\R^{n-k} \times\R}
\ind_{\{x_1<\ldots <x_{n-k} < x_n\}}  \cdot
\\
\cdot f(x_n) \varphi^k(x_n)\dd x_n \prod_{i=1}^{n-k} \varphi(x_i) \dd x_i.
\end{multline*}
Keeping $x_n$ fixed and integrating over the variables $x_1,\ldots,x_{n-k}$ yields
$$
\lim_{\eps \downarrow 0}\eps^{1-k} \E [f(\xi_{(n)}) \ind_{A_{n,k}(\eps)}]
=
\frac{n!}{(k-1)!(n-k)!}  \int_{\R}
f(x_n) \varphi^k(x_n)  \Phi^{n-k}(x_n) \dd x_n,
$$
where the factor $1/(n-k)!$ appeared because the variables $x_1,\ldots,x_{n-k}$ were ordered increasingly. Replacing $x_n$ by $s$ yields the required formula.

To justify the use of the dominated convergence theorem, observe that the function under the sign of the integral in~\eqref{eq:integral_dominated} can be estimated from above by
$$
\ind_{\{x_1<\ldots <x_{n-k} < x_n\}} \ind_{\{0 < s_1 <\ldots < s_{k-1} <1\}}  |f(x_n)| \varphi(x_n) \, \prod_{i=1}^{n-k} \varphi(x_i)
$$
because $\varphi(x_n -\eps s_{i})\leq 1$. The  integral of the latter function equals
$$
\frac{1}{(k-1)!(n-k)!} \int_{\R} |f(x_n)| \varphi(x_n) \Phi^{n-k}(x_n)\dd x_n,
$$
which is finite by the assumption $\int_{\R} |f(x_n)| \varphi(x_n)\dd x_n< \infty$.

\section{Asymptotics}\label{sec:asymptotics}
Since the work of Ruben~\cite{ruben,ruben_moments}, Hadwiger~\cite{hadwiger}, Efron~\cite{efron}, Affentranger and Schneider~\cite{AS92} it is known that integrals of the form $\int_{-\infty}^{\infty} \varphi^\alpha(t) \Phi^n(t) \dd t$ appear in the explicit formulae for the volumes of spherical regular simplices, external angles of regular simplices, expected volumes and number of faces of Gaussian polytopes. See also~\cite{kabluchko_zaporozhets_gauss_polytope} for further examples.
The following asymptotic equivalence  was derived in~\cite[pp.~44-45]{raynaud} and~\cite[Lemma~5]{vershik_sporyshev_asymptotic}:
$$
\int_{-\infty}^{\infty} \varphi^\alpha(t) \Phi^n(t) \dd t \sim n^{-\alpha} (2\log n)^{(\alpha-1)/2} \Gamma(\alpha), \quad n\to\infty.
$$
We shall prove a more precise result and then use it to deduce Theorem~\ref{theo:asympt_E_Vol}.
\begin{theorem}\label{theo:asympt_integral}
Let $u_n$ be given by~\eqref{eq:def_u_n}. If $n\to\infty$ while $\alpha>0$ stays fixed, we have
\begin{align}
\int_{-\infty}^{\infty} \varphi^\alpha(t) \Phi^n(t) \dd t
&=
 \frac{u_n^{\alpha-1} \Gamma(\alpha) + u_n^{\alpha-3}(\alpha-1) (\Gamma(\alpha) - \Gamma'(\alpha))  + o(u_n^{\alpha-3})}{n^\alpha},\label{eq:int_varphi_Phi_asympt1}\\
\int_{0}^{\infty} \varphi^{\alpha}(t) (2\Phi(t)-1)^{n} \dd t
&=
\frac {u_{2n}^{\alpha-1} \Gamma(\alpha) + u_{2n}^{\alpha-3}(\alpha-1) (\Gamma(\alpha) - \Gamma'(\alpha))  + o(u_{2n}^{\alpha-3})} {(2n)^{\alpha}}. \label{eq:int_varphi_Phi_asympt2}
\end{align}
\end{theorem}
\begin{proof}
We prove~\eqref{eq:int_varphi_Phi_asympt1} since the proof of~\eqref{eq:int_varphi_Phi_asympt2} is similar. Let $v_n$ be the solution of 
$$
\frac 1 {\sqrt {2\pi} v_n} \eee^{-v_n^2/2} = \frac 1n.
$$
Then, we have $v_n = u_n + o(\frac 1{\sqrt {2\log n}})$ and since $u_n\sim v_n\sim \sqrt{2\log n}$, we obtain
$$
v_{n}^{\alpha-1} = u_n^{\alpha-1} + o(u_n^{\alpha-3}), \quad  v_{n}^{\alpha-3} = u_n^{\alpha-3} + o(u_n^{\alpha-3}),
\quad n\to\infty.
$$
Hence, we can replace $u_n$ by $v_n$ in~\eqref{eq:int_varphi_Phi_asympt1}, and it suffices to prove the asymptotics
$$
\int_{-\infty}^{\infty} \varphi^\alpha(t) \Phi^n(t) \dd t = n^{-\alpha} \left(v_n^{\alpha-1} \Gamma(\alpha) + v_n^{\alpha-3}(\alpha-1) (\Gamma(\alpha) - \Gamma'(\alpha))  + o(v_n^{\alpha-3})\right),
$$
as $n\to\infty$.  Using the change of variables  $t= v_n + z/v_n$ and recalling that $\varphi(t) = \frac 1 {\sqrt{2\pi}}\eee^{-t^2/2}$, we can write
$$
\int_{-\infty}^{\infty} \varphi^\alpha(t) \Phi^n(t) \dd t
=
n^{-\alpha} v_n^{\alpha-1} \int_{-\infty}^{\infty} \eee^{-\alpha z} \eee^{-\frac{\alpha z^2}{2 v_n^2}}
\eee^{n \log \Phi\left(v_n +\frac z {v_n}\right)} \dd z.
$$
Next we use the following asymptotic expansion of the standard normal tail function
$$
1 - \Phi(t) = \frac 1 {\sqrt {2\pi} t } \eee^{-t^2/2} \left(1-\frac 1 {t^2} + O\left(\frac 1 {t^4}\right)\right),
\quad t\to +\infty.
$$
Inserting $t= v_n + z /v_n$, we arrive at
\begin{align}
1 - \Phi\left(v_n + \frac z {v_n}\right)
&=
\frac 1 {\sqrt {2\pi} \left(v_n + \frac z {v_n}\right) } \eee^{-v_n^2/2} \eee^{-z} \eee^{-\frac {z^2}{2v_n^2}} \left(1-\frac 1{v_n^2} + O\left(\frac 1 {v_n^4}\right)\right)\notag\\
&=\frac {\eee^{-z}} {n} \left(1- \frac {z^2 + 2z + 2}{2v_n^2} + O\left(\frac 1 {v_n^4}\right)\right),
\quad n\to\infty. \label{eq:Phi_bar_asympt}
\end{align}
Using the Taylor expansion  $\log x = x-1 + O((x-1)^2)$  as $x\to 1$, we get
\begin{equation}\label{eq:n_log_Phi}
n \log \Phi\left(v_n +\frac z {v_n}\right) = -\eee^{-z} \left(1- \frac {z^2 + 2z + 2}{2v_n^2}\right) + O\left(\frac 1 {v_n^4}\right).
\end{equation}
Taking everything together and using the Taylor expansion of the exponential, we arrive at
\begin{multline*}
\int_{-\infty}^{\infty} \varphi^\alpha(t) \Phi^n(t) \dd t
\\=
n^{-\alpha} v_n^{\alpha-1} \int_{-\infty}^{\infty} \eee^{-\alpha z} \eee^{-\eee^{-z}}
\left(1 -\frac{\alpha z^2}{2 v_n^2} + \eee^{-z}\left(\frac {z^2 + 2z + 2}{2v_n^2}\right) + O\left(\frac 1 {v_n^4}\right)\right) \dd z.
\end{multline*}
It remains to use the integral
$$
\int_{-\infty}^{\infty} z^k  \eee^{-\eee^{-z}} \eee^{-\alpha z} \dd z = (-1)^k\Gamma^{(k)}(\alpha),
$$
for $k=0,1,2$, to get
\begin{multline*}
\int_{-\infty}^{\infty} \varphi^\alpha(t) \Phi^n(t) \dd t
\\=
n^{-\alpha} v_n^{\alpha-1} \left(\Gamma(\alpha) + \frac 1 {2v_n^2} \left(-\alpha \Gamma''(\alpha) + \Gamma''(\alpha+1) - 2 \Gamma'(\alpha+1) + 2\Gamma(\alpha+1)\right) + O\left(\frac 1 {v_n^4}\right) \right).
\end{multline*}
Using the relations $\Gamma(\alpha+1) = \alpha \Gamma(\alpha)$, $\Gamma'(\alpha+1) = \Gamma(\alpha) + \alpha \Gamma'(\alpha)$ and $\Gamma''(\alpha+1) = 2\Gamma'(\alpha)+ \alpha \Gamma''(\alpha)$, we obtain the required asymptotics. We omitted the justification of the interchanging the integral and the limit because it is standard.  
\end{proof}

\begin{proof}[Proof of Theorem~\ref{theo:asympt_E_Vol}]
Recall that $\E \Vol_d (\cP_{n,d})$ and $\E \Vol_d (\cP_{n,d}^\pm)$ are given by~\eqref{eq:exp_vol_P_n_d} and~\eqref{eq:exp_vol_P_n_d_symm}, respectively. Apply Theorem~\ref{theo:asympt_integral} with $\alpha=d+1$ to the integrals appearing in these formulae. For example,
$$
\E \Vol_d (\cP_{n,d}) = \frac {\kappa_d}{d!} \frac{n!\, n^{-d-1}}{(n-d-1)!} \left(u_{n-d-1}^d d! + d u_{n-d-1}^{d-2}(\Gamma(d+1) - \Gamma'(d+1)) + o(u_{n-d-1}^{d-2})\right).
$$
Notice that 
$$
\Gamma(d+1) - \Gamma'(d+1) = d!\left(\gamma - \sum_{j=2}^d \frac 1j\right).
$$
Further, observe that $u_{n-d-1} = u_n + o(u_n^{-1})$ implies that
$$
u_{n-d-1}^{d} = u_n^d + o(u_n^{d-2}), \quad u_{n-d-1}^{d-2} = u_n^{d-2} + o(u_n^{d-4}).
$$
Also, $\frac {n!}{(n-d-1)!} = n^{d+1}(1+O(\frac 1n))$. 
Taking everything together, we obtain the required formula~\eqref{eq:E_Vol_asympt_simplex}. 
\end{proof}

Finally, let us state a ``poissonized'' version of Theorem~\ref{theo:asympt_integral}.
\begin{theorem}\label{theo:asympt_integral_poi}
Let $u_\lambda$ be given by~\eqref{eq:def_u_n} with $n$ replaced by $\lambda$. If $\alpha>0$ is fixed and $\lambda\to\infty$, then
\begin{align}
\int_{-\infty}^{\infty} \varphi^\alpha(t) \eee^{\lambda (\Phi(t)-1)} \dd t
&=
\frac {1}{\lambda^{\alpha}} \left(u_\lambda^{\alpha-1} \Gamma(\alpha) + u_\lambda^{\alpha-3}(\alpha-1) (\Gamma(\alpha) - \Gamma'(\alpha))  + o(u_\lambda^{\alpha-3})\right).\label{eq:int_varphi_Phi_asympt1_poi}
\end{align}
\end{theorem}
\begin{proof}
One can prove~\eqref{eq:int_varphi_Phi_asympt1_poi} in the same way as~\eqref{eq:int_varphi_Phi_asympt1}, but instead of the term $\eee^{n\log \Phi (v_n + \frac z {v_n})}$ we now have $\eee^{\lambda \Phi(v_\lambda + \frac {z}{v_\lambda}) - \lambda}$. Using~\eqref{eq:Phi_bar_asympt} with $n$ replaced by $\lambda$, we can write
$$
\lambda \Phi\left(v_\lambda + \frac {z}{v_\lambda}\right) - \lambda = - \eee^{-z} \left(1- \frac {z^2 + 2z + 2}{2v_\lambda^2} + O\left(\frac 1 {v_\lambda^4}\right)\right),
\quad \lambda\to\infty.
$$
This replaces~\eqref{eq:n_log_Phi}. The rest of the proof is the same.
\end{proof}

\begin{remark}\label{rem:same_asymptotics}
It is easy to see that the asymptotics in~\eqref{eq:int_varphi_Phi_asympt1} and~\eqref{eq:int_varphi_Phi_asympt1_poi} do not change if we replace $\int_{-\infty}^{\infty}$ by $\int_C^{\infty}$ (for any constant $C$) since the main contribution to the integral  comes from the neighborhood of the point $v_n$, respectively $v_\lambda$.
\end{remark}

Starting with the formulae~\eqref{eq:exp_vol_P_n_d_poi} and~\eqref{eq:exp_vol_P_n_d_symm_poi} and applying Theorem~\ref{theo:asympt_integral_poi} one easily gets~\eqref{eq:E_Vol_asympt_simplex_poi} and~\eqref{eq:E_Vol_asympt_cross_poi} with the same arguments as in the above proof of Theorem~\ref{theo:asympt_E_Vol}.

\bibliographystyle{plainnat}
\bibliography{intrinsic_volumes_bib}

\end{document}